\documentclass[a4paper,12pt,leqno]{article}
\usepackage{amsmath}
\usepackage{amsthm}
\usepackage{amsfonts}
\usepackage{amssymb}
\usepackage{graphics}
\usepackage{color}

\usepackage{txfonts}

\makeatletter
\def\labelenumi{\rm (\makebox[8pt]{\@arabic\c@enumi})}
\def\theenumi{\@arabic\c@enumi}
\makeatother
\theoremstyle{theorem}
\newtheorem{theorem}{Theorem}

\newtheorem{lemma}{Lemma}

\newtheorem{lemma-app}{Lemma A.\!}

\newtheorem{mainresult}{Main result}

\theoremstyle{definition}
\newtheorem{definition}{Definition}
\newtheorem{remark}{Remark}

\def\Res{\mathop{\rm Res}}
\def\sgn{\mathop{\rm sgn}}
\def\supp{\mathop{\rm supp}}

\newcommand{\Sch}[1]{\ensuremath{\mathcal{S}(#1)}}

\makeatletter
\def\multilimits@{\bgroup
  \Let@
  \restore@math@cr
  \default@tag
 \baselineskip\fontdimen10 \scriptfont\tw@
 \advance\baselineskip\fontdimen12 \scriptfont\tw@
 \lineskip\thr@@\fontdimen8 \scriptfont\thr@@
 \lineskiplimit\lineskip
 \vbox\bgroup\ialign\bgroup\hfil$\m@th\scriptstyle{##}$\hfil\crcr}
\def\Sb{_\multilimits@}
\def\Sp{^\multilimits@}
\def\endSb{\crcr\egroup\egroup\egroup}

\def\smallmatrix{\null\,\vcenter\bgroup
 \Let@\restore@math@cr\default@tag
 \baselineskip6\ex@ \lineskip1.5\ex@ \lineskiplimit\lineskip
 \ialign\bgroup\hfil$\m@th\scriptstyle{##}$\hfil&&\thickspace\hfil
 $\m@th\scriptstyle{##}$\hfil\crcr}
\def\endsmallmatrix{\crcr\egroup\egroup\,}
\newcount\c@MaxMatrixCols
\makeatother

\def\Z{\mathbb{Z}}
\def\Q{\mathbb{Q}}
\def\R{\mathbb{R}}
\def\C{\mathbb{C}}
\def\bfe{\mathbf{e}}

\begin{document}
\title{The modularity of Siegel's zeta functions}
\author{Kazunari Sugiyama
\!\!\footnote{Department of Mathematics, Chiba Institute of Technology, 2-1-1 
Shibazono, Narashino, Chiba, 275-0023, Japan. 
E-mail:\texttt{skazu@sky.it-chiba.ac.jp}}
\,\footnote{\textbf{Funding}: This research is supported by JSPS KAKENHI 
Grant Number 22K03251. 
\textbf{Data Availability}: Data sharing not applicable to this article as no datasets were generated or analysed during the current study.
\textbf{Conflict of interest}: The author declares that there are no conflicts of interest.}
}
\maketitle

\begin{abstract}
Siegel defined zeta functions associated with indefinite quadratic forms, and proved their analytic properties
such as analytic continuations and functional equations. 
Coefficients of these zeta functions are called measures of representations, and play an important role in the arithmetic 
theory of quadratic forms. 
In a 1938 paper, Siegel made a comment
 to the effect that the modularity of his zeta functions
would be proved with the help of
a suitable converse theorem. In the present paper, we accomplish
Siegel's original plan by using a Weil-type converse theorem for Maass forms, 
which has appeared recently.
It is also shown 
that ``half'' of Siegel's zeta functions correspond to holomorphic modular forms. 
\end{abstract}

\section*{Introduction}

In 1903, Epstein~\cite{E} defined the zeta function $\zeta_0(s)$
associated with a positive definite symmetric matrix $S$ of degree $m$ by
\[
 \zeta_0(s)= \sum_{a\in \Z^{m}\setminus\{0\}} \frac{1}{S[a]^s} \qquad 
 (S[a]= {}^{t}a S a),
\]
and studied their analytic properties such as analytic continuations and
functional equations.  
(For a modern treatment of Epstein's zeta functions, we refer to
Terras~\cite[\S 1.4.2]{Terras}.)
In a 1938 paper~\cite{SiegelZF1}, 
Siegel defined and investigated the zeta functions associated with
quadratic forms of signature $(1,m-1)$, and
in a 1939 paper~\cite{SiegelZF2}, those 
for general indefinite quadratic forms.
Although Siegel's calculations were rather involved, 
Siegel's results are now well understood in 
the framework of the theory of prehomogeneous vector spaces.
Let $Y$ be a non-degenerate half-integral symmetric matrix of degree $m$
with $p$ positive eigenvalues and $m-p$ negative eigenvalues 
($0<p<m$). 
Let $SO(Y)$ be the special orthogonal group of $Y$ and denote by 
$SO(Y)_{\Z}$ its arithmetic subgroup. We put
$V_{\pm}=\{v\in \R^m\, ; \sgn Y[v]=\pm\}$. Then Siegel's zeta functions 
are Dirichlet series associated with the prehomogeneous vector space
$(GL_1(\C)\times SO(Y), \, \C^m)$, and 
are defined by
\[
 \zeta_{\pm}(s)=
 \sum_{v\in SO(Y)_{\Z}\backslash (\Z^m\cap V_{\pm})}
 \frac{\mu(v)}{|Y[v]|^s},
\]
where the sum runs over a complete set of representatives of
$SO(Y)_{\Z}\backslash (\Z^m \cap V_{\pm})$, 
and $\mu(v)$  is a certain volume of the fundamental domain 
related to the isotropy subgroup $SO(Y)_v$ of $SO(Y)$ at $v$. 
In the positive definite case, the {\it modularity} of Epstein's zeta
function $\zeta_0(s)$ is almost obvious; $\zeta_0(s)$ 
is obtained by taking the Mellin transform
of (the restriction to the imaginary axis of) the theta series
\[
 \theta(S, z)= \sum_{a\in \Z^m} e^{\pi i S[a] z}
 \qquad (z= x+ iy\in \C, \, y>0),
\]
which is a modular form for a subgroup of $SL_2(\Z)$. 
(cf.\ Miyake~\cite[\S 4.9]{Miyake}, \ Terras~\cite[\S 3.4.4]{Terras}.)
On the contrary, it is not clear from the definition whether or not
Siegel's zeta function arises as an integral transform of some infinite series
with modular properties. Rather,  in the preface to a 1938 paper~\cite{SiegelZF1}, 
Siegel wrote that 
such theta series would be constucted from his zeta functions, 
citing the work of 
Hecke~\cite{HeckeIndefQ1925}, in which Hecke derived the
transformation formula for the theta series associated with indefinite binary 
quadratic forms from the functional equation of zeta functions of
real quadratic fields.
Furthermore, Siegel made the following remark
in the last section of~\cite{SiegelZF1}:

\begin{quotation}
{\it  Will man die Transformationstheorie von $f(\mathfrak{S}, x)$
 f\"{u}r beliebige Modulsubstitutionen entwickeln, so hat man au\ss{}er
 $\zeta_1(\mathfrak{S}, s)$ auch analog gebildete Zetafunktionen mit Restklassen-Chrakteren
 zu untersuchen. Die zum Beweise der S\"{a}tze 1, 2,3 f\"{u}hrenden 
  \"{U}berlegungen lassen sich ohne wesentiche Schwierigkeit auf den allgemeinen Fall 
  \"{u}bertragen. Verm\"{o}ge der Mellinschen Transformation erh\"{a}lt man dann 
  das wichtige Resultat, da\ss{} die durch (53) definierte Funktion $f(\mathfrak{S}, x)$
  eine Modulform der Dimension $\frac{n}{2}$ und der Stufe $2D$ ist; dabei wird
  vorausgesetzt, da\ss{} $n$ ungerade und $\mathfrak{x}'\mathfrak{S}\mathfrak{x}$
  keine tern\"{a}re Nullform ist. }
\end{quotation}

\begin{quotation}
If one wants to develop the transformation theory of $f(\mathfrak{S},x)$
 for arbitrary modular substitutions, then in addition to 
 $\zeta_1(\mathfrak{S},s)$ one also has to investigate zeta functions
 formed analogously  with residual class characters.
 The considerations leading to the proof of Theorems 1, 2, 3 
 can be transferred to the general case without any major difficulty. 
 By virtue of the (inverse) Mellin transformation, one then obtains
 an important result that the function  $f(\mathfrak{S}, x)$
 defined by (53) is a modular form of weight $\frac{n}{2}$ and level $2D$, 
 provided that  $n$ is odd and $\mathfrak{x}'\mathfrak{S}\mathfrak{x}$
 is not a ternary zero form.
\end{quotation}

As of 1938, Siegel seemed to have noticed the possibility that 
by considering the twists of zeta functions by Dirichlet characters, one 
can prove modularity for congruence subgroups. 
In the holomorphic case,
this fact is known as Weil's converse theorem~\cite{Weil}.
It was 1967 when Weil's paper~\cite{Weil} appeared!
Revisiting Siegel's prediction in the light of recent developments 
is one motivation for the present study.

We should note that in the quotation above, Siegel mentioned
the parity of $n$, 
the number of variables of quadratic forms.
This is related to the fact that the concept of non-holomorphic modular forms was 
not yet in place at that time. 
In a celebrated paper~\cite{Maass}, 
Maa\ss{} introduced the notion of the so-called Maass forms
and established a Hecke correspondence for Maass forms. 
Further, in \cite{MaassIQF}, as its application of his theorem, 
Maa\ss{}  proved that in a very special case (when $Y$ is diagonal of even degree
with $\det Y=1$), Siegel's zeta functions can be expressed as the product
of two standard Dirichlet series such as the Riemann zeta function $\zeta(s)$
and the Dirichlet $L$-function $L(s, \chi)$. 
On the other hand, it is only recently
that papers on Weil-type converse theorems for Maass forms have  
emerged (cf.\ \cite{MSSU, NeururerOliver}).
It would be a very natural idea for us
to accomplish Siegel's original plan
to prove the modularity via converse theorem including the case of
non-holomorphic  forms.

Siegel's zeta functions are closely related to the so-called
Siegel's main theorem ({\it Siegelsche HauptSatz}).
In a 1951 paper~\cite{Siegel}, Siegel proved
the transformation formula for some theta series arising from indefinite quadratic forms, and
the equality between an integral of the indefinite theta series 
over fundamental domains and some
 Eisenstein series (cf. \cite[Satz~1]{Siegel}).
It was  shown in \cite[Hilfssatz 4]{Siegel} that  the coefficients 
\[
 M(Y;\pm n) = \sum
 \begin{Sb}
 v\in SO(Y)_{\Z}\backslash (\Z^m \cap V_{\pm})\\
 Y[v]= \pm n
 \end{Sb} \mu(v) \qquad (n=1, 2, 3,\dots)
\]
of $\zeta_{\pm}(s)$ coincide with the Fourier coefficients 
of the non-holomorphic modular forms appearing in Siegel's formula. 
Here we ignore the differences in the definitions of $\mu(v)$; 
the definitions of measures are different for each of the  
papers~\cite{SiegelZF1, SiegelZF2, Siegel}.
Siegel called $M(Y;n)$ 
the measures of representations ({\it Darstellungsma\ss{}}). 
The measure $M(Y;n)$ of representations is an analogue of the 
representation number 
\[
r(S;n)=\sharp \{a\in \Z^{m}\; ;\; S[a]=n\}
\]
for a positive symmetric matrix $S$, and Siegel's formula can be
reformulated as  an arithmetic 
identity that $M(Y, n)$ is equal to the product of local representation
densities over all primes. 
Weil~\cite{WeilActa1965} generalized Siegel's result by using the language 
of adeles, and it is the Siegel-Weil formula---a cornerstone in the modern
 number theory.

Now we explain the main results of the present paper. 
First, along the Sato-Shintani theory~\cite{SatoShintani} of 
prehomogeneous vector spaces, we define Siegel's zeta functions
and prove their analytic properties.  
Here, to treat twisted zeta functions as well as the original 
Siegel's zeta functions, we first consider 
Siegel's zeta functions {\it with congruence conditions}, 
which are defined using Schwartz-Bruhat functions
on $\Q^m$. This idea is due to F.\ Sato~\cite{SatoZetaDist}.
Then the converse theorem in~\cite{MSSU} is applied to the zeta functions,
and the following result  is obtained:

\begin{mainresult}[Theorem~\ref{theorem:SiegelMaass}]
Let $m\geq 5$.
Assume that at least one of $m$ or $p$ is an odd integer.
Take an integer $\ell$
with $\ell\equiv 2p-m \pmod{4}$, and put $D=\det (2Y)$. 
Let $N$ be the level of $2Y$. 
Define $C^{\infty}$-function 
$F(z)= F(x+iy)$ on the Poincar\'{e} upper half-plane $\mathcal{H}$ by
\begin{align*}
F(z) &=   y^{(m-\ell)/4}\cdot \int_{SO(Y)_{\R}/SO(Y)_{\Z}} d^{1}g  
+ \alpha(0) \cdot  y^{1-(m+\ell)/4}
\\
&+\sum
\begin{Sb}
n=-\infty \\
n\neq 0
\end{Sb}^{\infty}
 (-1)^{(2p-m-\ell)/4}\cdot  \frac{M(Y;n)}{|D|^{\frac{1}{2}}}
\frac{\pi^{\frac{m}{4}} \cdot  |n|^{-\frac{m}{4}}}
{\Gamma\left(\frac{m+\sgn(n)\ell}{4}\right)}\cdot
y^{-\frac{\ell}{4}}W_{\frac{\sgn(n)\ell}{4}, \frac{m}{4}-\frac{1}{2}}(4\pi |n|y) \bfe[nx],
\end{align*}
where $d^{1}g$ is a suitably normalized Haar measure on $SO(Y)_{\R}$, 
$\alpha(0)$ is some constant determined by the residues of $\zeta_{\pm}(s)$,
and $W_{\mu, \nu}(y)$ denotes the Whittaker function. Then, 
$F(z)$ is a Maass form of weight $\ell/2$ with respect to 
 the congruence subgroup $\Gamma_0(N)$. 
\end{mainresult}
The above formula can be compared with Siegel's 
calculation~\cite[Hilfssatz~4]{Siegel} of the Fourier expansions
of non-holomorphic modular forms. 
Our $F(z)$ is essentially the same as the modular form given by Siegel~\cite{Siegel}.
See Remark~\ref{remark:RelationToSiegel}.
The theorem above excludes the case where both $m$ and $p$ are even. 
Our second result states that if one of $m-p$ and $p$ are even, 
we can construct 
holomorphic modular forms from $M(Y; \pm n)$.

\begin{mainresult}[Theorem~\ref{theorem:HolomorphicForms}]
Let $m\geq 5$.
Assume that $m-p$ is even. 
We define a holomorphic function $F(z)$ on $\mathcal{H}$ by
\begin{align*}
F(z) &= (-1)^{\frac{m-p}{2}} (2\pi)^{-\frac{m}{2}} \cdot \Gamma\left(\frac{m}{2}\right) 
\int_{SO(Y)_{\R}/SO(Y)_{\Z}} d^{1}g \\
&\qquad 
 +  |D|^{-1/2}\cdot\sum_{n=1}^{\infty}   M(Y;n)\mathbf{e}[nz].
\end{align*}
Then, $F(z)$ is a holomorphic modular form of weight $m/2$ with respect to 
$\Gamma_0(N)$.
(In the case that $p$ is even, we can construct holomorphic modular forms
from $M(Y; -n)\; (n=1,2,3,\dots)$. )
\end{mainresult}

The theorem above is consistent with a result of Siegel
that was published  in a 1948 paper~\cite{Siegel1948}.
In this paper, Siegel calculated the action of certain differential operators
on indefinite theta series, and proved that in the case of $\det Y>0$, 
we can construct holomorphic modular forms from indefinite theta series
associated with $Y$.

Before closing Introduction, we give some
 remarks on related researches, future problems, 
 and possible applications.
Special values of 
Siegel's zeta functions associated with $Y$ of signature $(1, m-1)$ appear
in the dimension formula for automorphic forms on orthogonal 
groups of signature $(2, m)$ (cf. Ibukiyama~\cite{IbukiyamaRIMS95}), and
it is important to investigate their arithmetic aspects.
Ibukiyama~\cite{Ibukiyama} proved an explicit formula expressing 
Siegel zeta functions (with $m$ even) 
as linear combinations of products of 
two shifted Dirichlet $L$-functions and certain elementary factors.
His proof is given by direct calculations using Siegel's main 
theorem in~\cite{Siegel} and not by converse theorems.  
Ibukiyama's explicit formula is quite general
and includes the above-mentioned result of Maa\ss{}~\cite{MaassIQF}.
We also mention the work~\cite{HafWal} of Hafner-Walling, in which 
they carried out extensive calculations to make Siegel's formula more explicit 
in terms of standard Eisenstein series. This work is also restricted to the case 
where $m$ is even. 
It is worthwhile to investigate the case where $m$ is odd. 
Finally, in a good situation, the method of converse theorems can be used 
to prove lifting theorems. In~\cite{SugiyamaPJAS}, a Shintani-Katok-Sarnak 
type correspondence is derived from analytic properties of a 
certain prehomogeneous zeta function whose coefficients involve
periods of Maass cusp forms. 
In~\cite{Maass1959}, Maa\ss{}  
studied a generalization of Siegel's zeta functions, which can be
regarded as prehomogeneous zeta functions
whose coefficients involve periods of automorphic forms on orthogonal groups.
It is quite probable that  our method can be applied to these zeta 
functions, and some lifting theorems will be obtained.  
We hope to discuss this topic elsewhere.

The present paper is organized as follows.
In Section~\ref{section:Preliminaries}, 
we recall a Weil-type converse theorem for Maass forms, and 
in Section~\ref{section:PV}, we define our prehomogeneous vector spaces 
and give the local functional equations. 
Section~\ref{section:SiegelZeta} is devoted to define Siegel's zeta functions
with congruence conditions, 
and analytic properties of Siegel's zeta functions are proved in 
Section~\ref{section:PropertiesOfSiegelZeta}.
We prove our main theorems in 
Sections~\ref{section:MainTh} and~\ref{section:Holomorphic}.

\bigskip

\noindent
\textbf{Acknowledgement.}
The author wishes to thank Professor Fumihiro Sato for stimulating
discussion and helpful suggestions. 
The author also thanks Professor Tomoyoshi Ibukiyama for valuable comments; 
in particular, Professor Ibukiyama  explained the relation 
of the results of this paper to
the prior work of~\cite{IbukiyamaRIMS95, Ibukiyama, MaassIQF, Siegel1948}.
Finally, the author would like to thank anonymous reviewers for 
their careful reading and helpful comments.

\bigskip

\noindent
\textbf{Notation.}
We denote by $\Z, \Q,\R$, and $\C$ the ring of integers, the field of rational numbers,  the field of real numbers, and  the field of complex numbers, respectively. 
The set of non-zero real numbers and the set of positive real numbers are denoted by $\R^{\times}$ and $\R_{+}$, respectively.
The set of positive integers and the set of non-negative integers are denoted by $\Z_{>0}$ and $\Z_{\geq0}$, respectively. 
The real part and the imaginary part of a complex number $s$ are denoted by $\Re(s)$ and $\Im(s)$, respectively. 
For complex numbers $\alpha, z$ with $\alpha\ne0$, $\alpha^z$ always stands for the principal value, namely, $\alpha^z = \exp((\log|\alpha|+i\,\mathrm{arg}\, \alpha)z)$ with $-\pi<\mathrm{arg}\, \alpha\leq\pi$. 
We use $\mathbf{e}[x]$ to denote $\exp(2\pi i x)$. 
The quadratic residue symbol $\left(\frac{\ast}{\ast}\right)$ has the same meaning as in Shimura~\cite[p.~442]{Shimura73}. 
For a meromorphic function $f(s)$ with a pole at $s=\alpha$, we denote its residue at $s=\alpha$ by $\displaystyle \Res_{s=\alpha} f(s)$.

\section{A Weil-type converse theorem for Maass forms}
\label{section:Preliminaries}

In this section, we define Maass forms
on the Poincar\'{e} upper half-plane
$\mathcal{H}=\{z\in \C\, | \, \Im(z)>0\}$
of integral and half-integral weight, and 
recall a Weil-type converse theorem for Maass forms 
that is proved in \cite{MSSU}.
We refer to Cohen-Str\"{o}mberg~\cite{CH} for an overview of 
the theory of Maass forms. 
Let 
$\Gamma=SL_{2}(\Z)$ be the modular group, and
for a positive integer $N$, we denote by $\Gamma_0(N)$ the
congruence subgroup defined by
\[
 \Gamma_0(N)=\left\{ \left.
 \begin{pmatrix}
 a & b \\
 c & d 
 \end{pmatrix}\in \Gamma \,  \right| \,  
 c\equiv 0 \pmod{N}
 \right\}.
\]
As usual, $\Gamma$ acts on $\mathcal{H}$ by
the linear fractional transformation
\[
 gz = \frac{az+b}{cz+d} \quad \text{for}
 \quad g =
 \begin{pmatrix}
 a & b \\
 c & d 
 \end{pmatrix} \in \Gamma.
\]
We put $j(\gamma, z)=cz+d$, and define $\theta(z)$ and
$J(\gamma, z)$ by
\[
\theta(z)=\sum_{n=-\infty}^{\infty} \exp(2\pi i n^2 z), \qquad
J(\gamma, z)=\frac{\theta(\gamma z)}{\theta(z)}.
\]
Then it is well-known that
\[
J(\gamma, z)= \varepsilon_{d}^{-1} \cdot \left(\frac{c}{d}\right)\cdot 
 (cz+d)^{1/2} \quad \text{for}\; \; \gamma=
 \begin{pmatrix}
 a & b \\
 c & d 
 \end{pmatrix}\in \Gamma_0(4), 
\]
where 
\begin{equation}
\label{eqn:def of epsilond}
\varepsilon_d = \begin{cases}
 1 & (d \equiv 1 \pmod 4), \\
 i & (d \equiv 3 \pmod 4).
 \end{cases}
\end{equation}

For an integer $\ell$, 
the hyperbolic Laplacian 
$\Delta_{\ell/2}$ of weight $\ell/2$  on $\mathcal{H}$ is defined by  
\begin{equation}
\label{form:Laplacian}
 \Delta_{\ell/2} = -y^2\left(\frac{\partial^2}{\partial x^2}+
 \frac{\partial^2}{\partial y^2}\right)+ \frac{i \ell y}{2}\left(\frac{\partial}{\partial x}
 +i\frac{\partial}{\partial y}\right) \quad (z=x+iy \in \mathcal{H}). 
\end{equation}

Let $\chi$ be a Dirichlet character mod ${N}$. Then we use 
the same symbol $\chi$ to denote the character of ${\Gamma}_0(N)$ 
defined by
\begin{equation}
 \label{form:ChiInducesCharOnGamma0N}
 \chi(\gamma)=\chi(d) \qquad \text{for} \quad
 \gamma=
 \begin{pmatrix}
 a & b \\
 c & d 
 \end{pmatrix}\in \Gamma_0(N).
\end{equation}

\begin{definition}[Maass forms]
\label{defn:DefOfMaassForms}
Let $\ell \in \Z$, and $N$ be a positive integer, with $4|N$ when $\ell$ is odd.  
A complex-valued $C^{\infty}$-function $F(z)$ on $\mathcal{H}$
is called {\it a Maass form for} ${\Gamma}_0(N)$
of weight $\ell/2$ with character $\chi$, if the following three conditions 
are satisfied;
\begin{enumerate}
\def\labelenumi{(\roman{enumi})}
\item  for every $\gamma \in \Gamma_0(N)$, 
\[
 F(\gamma z) = 
 \begin{cases}
 \chi(\gamma) j(\gamma, z)^{\ell/2} \cdot F(z) & 
 (\text{$\ell$ is even}) \\[3pt]
  \chi(\gamma) J(\gamma, z)^{\ell} \cdot F(z) & 
 (\text{$\ell$ is odd})
 \end{cases},
\]
\item 
$\Delta_{\ell/2} F= \Lambda \cdot F$ with some $\Lambda\in \C$,
\item  
$F$ is of moderate growth at every cusp,
namely,  for every $A \in SL_2(\Z)$, there exist positive constants $C$, $K$ and $\nu$ depending on $F$ and $A$ such that 
\[
|F(A z)| \cdot |j(A, z)|^{-\ell/2} < 
C y^{\nu} \quad \text{if}\;\;   y= \Im(z)>K.
\]
\end{enumerate}
We call $\Lambda$ the {\it eigenvalue} of $F$. 
\end{definition}

Let 
 $\lambda$ be a complex number with $\lambda \not\in 1-\frac12\Z_{\geq0}$. 
Let 
 $\alpha=\{\alpha(n)\}_{n\in \Z\setminus\{0\}}$ and $\beta=\{\beta(n)\}_{n\in \Z\setminus\{0\}}$ be  complex sequences of polynomial growth. 
For $\alpha,\beta$, we can define the  $L$-functions 
 $\xi_{\pm}(\alpha;s), \xi_{\pm}(\beta;s)$
  and the completed $L$-functions
 $\Xi_{\pm}(\alpha;s), \Xi_{\pm}(\beta;s)$ by
 \begin{align*}
 \xi_{\pm}(\alpha;s) &= 
 \sum_{n=1}^{\infty}\frac{\alpha(\pm n)}{n^s}, & 
 \Xi_{\pm}(\alpha;s)&=
 (2\pi)^{-s} \Gamma(s) \xi_{\pm}(\alpha;s), \\
  \xi_{\pm}(\beta;s) &= 
 \sum_{n=1}^{\infty}\frac{\beta(\pm n)}{n^s}, & 
 \Xi_{\pm}(\beta;s)&=
 (2\pi)^{-s} \Gamma(s) \xi_{\pm}(\beta;s).  
 \end{align*}

In the following, for simplicity, we put 
\begin{gather*}
\label{eqn:xi-alpha-e-o}
\xi_e(\alpha;s) = \xi_+(\alpha;s) + \xi_-(\alpha;s), \quad
\xi_o(\alpha;s) = \xi_+(\alpha;s) - \xi_-(\alpha;s).
\end{gather*}

Now we assume the following conditions [A1] -- [A4]: 

\begin{description}
 \item[{\bf [A1]}] The $L$-functions $\xi_{\pm}(\alpha;s), \xi_{\pm}(\beta;s)$  have meromorphic continuations to the whole $s$-plane, and $(s-1)(s-2+2\lambda)\xi_\pm(\alpha;s)$ and $(s-1)(s-2+2\lambda)\xi_\pm(\alpha;s)$ are entire functions, which are of finite order in any vertical strip. 
\end{description}
Here a function $f(s)$ on a vertical strip $\sigma_1 \leq \Re(s) \leq \sigma_2$ $(\sigma_1,\sigma_2 \in \R, \sigma_1<\sigma_2)$ is said to be {\it of finite order\/} on the strip if there exist some positive constants $A,B,\rho$ such that 
\[
|f(s)| < A e^{B|\Im(s)|^\rho}, \quad \sigma_1 \leq \Re(s) \leq \sigma_2.
\]
\begin{description}
 \item[{\bf [A2]}]
The residues of  $\xi_{\pm}(\alpha;s)$ and $\xi_{\pm}(\beta;s)$ at $s=1$ satisfy
\[
\mathop{\mathrm{Res}}_{s=1} \xi_+(\alpha;s) = \mathop{\mathrm{Res}}_{s=1} \xi_-(\alpha;s), \quad 
\mathop{\mathrm{Res}}_{s=1} \xi_+(\beta;s) = \mathop{\mathrm{Res}}_{s=1} \xi_-(\beta;s). 
\] 
 \item[{\bf [A3]}] The following functional equation holds:
 \[
  \gamma(s)\left(
 \begin{array}{c}
 \Xi_{+}(\alpha;s) \\[2pt] \Xi_{-}(\alpha;s)
 \end{array}
 \right)\\
 = N^{2-2\lambda-s}\cdot 
 \Sigma(\ell)\cdot
 \gamma(2-2\lambda-s)\left(
 \begin{array}{c}
 \Xi_{+}(\beta;2-2\lambda-s) \\[2pt] \Xi_{-}(\beta;2-2\lambda-s)
 \end{array}
 \right),
 \]
 where  $\gamma(s)$ and $\Sigma(\ell)$ are defined by
  \begin{equation}
  \label{form:DefOfGammaAndSigma}
  \gamma(s) = 
  \begin{pmatrix}
  e^{\pi s i/2}  & e^{-\pi s i/2} \\
  e^{-\pi s i/2}  & e^{\pi s i/2} 
 \end{pmatrix}, \qquad
 \Sigma(\ell) =
 \begin{pmatrix}
 0 & i^{\ell} \\
 1 & 0
 \end{pmatrix}.
 \end{equation}
\item[{\bf [A4]}]
If $\lambda = \frac q2$ $(q \in \Z_{\geq0},\ q \geq 4)$, then
\[
\xi_+(\alpha;-k)+(-1)^k\xi_-(\alpha;-k)=0 \quad (k = 1,2,\ldots,q-3).
\]
\end{description}

Under the assumptions [A1] -- [A4], we define $\alpha(0)$, $\beta(0)$, $\alpha(\infty)$, $\beta(\infty)$ by 
\begin{eqnarray}
 \label{form:DefOfA0AndAInfty}
 \alpha(0) &=&  
-\xi_{e}(\alpha;0) \\
 \label{form:DefOfA0AndAInfty2}
\alpha(\infty) &=&
 \frac{N}{2} \Res_{s=1}
 \xi_{e}(\beta;s), \\[4pt]
 \label{form:DefOfB0AndBInfty}
 \beta(0) &=&
  -\xi_{e}(\beta;0)   \\
 \label{form:DefOfB0AndBInfty2}
 \beta(\infty) &=&
 \frac{i^{-\ell}}{2} \Res_{s=1}
 \xi_{e}(\alpha;s),
\end{eqnarray}

For an odd prime number $r$ with $(N,r)=1$ and a Dirichlet character $\psi$ mod $r$, the twisted $L$-functions 
  $\xi_{\pm}(\alpha, \psi;s), \Xi_{\pm}(\alpha, \psi;s), 
    \xi_{\pm}(\beta, \psi;s), \Xi_{\pm}(\beta, \psi;s)$ 
are defined by
 \begin{align*}
 \xi_{\pm}(\alpha, \psi;s) &= 
 \sum_{n=1}^{\infty}\frac{\alpha(\pm n)\tau_{\psi}(\pm n)}{n^s}, & 
 \Xi_{\pm}(\alpha,\psi; s)&=
 (2\pi)^{-s} \Gamma(s) \xi_{\pm}(\alpha,\psi;s), \\
  \xi_{\pm}(\beta,\psi;s) &= 
 \sum_{n=1}^{\infty}\frac{\beta(\pm n)\tau_{\psi}(\pm n)}{n^s}, & 
 \Xi_{\pm}(\beta, \psi;s)&=
 (2\pi)^{-s} \Gamma(s) \xi_{\pm}(\beta,\psi;s), 
 \end{align*}
where $\tau_{\psi}(n)$ is the Gauss sum defined by 
\begin{equation}
   \tau_{\psi}(n)= \sum
 \begin{Sb}
 m \bmod r \\
 (m, r)=1 
 \end{Sb} \psi(m) e^{2\pi i mn/r}. 
\label{eqn:def of Gauss sum}
\end{equation}
We put 
\[
\tau_\psi=\tau_\psi(1)=\sum
 \begin{Sb}
 m \bmod r \\
 (m, r)=1 
 \end{Sb} \psi(m) e^{2\pi i m/r}
\]
and denote by $\psi_{r,0}$ the principal character modulo $r$.  
Recall that the Gauss sums are calculated as follows:
\begin{eqnarray}
\tau_\psi(n) 
 &=& \begin{cases}
   \overline{\psi(n)} \tau_\psi  & (n \not\equiv 0 \pmod{r}), \\
   0 &  (n \equiv 0 \pmod{r}),
      \end{cases}
      \quad \text{if $\psi\ne\psi_{r,0}$},  
\label{eqn:Gauss sum non-trivial psi}\\ 
\tau_{\psi_{r,0}}(n) 
  &=& \begin{cases}
  -1 &  (n \not\equiv 0 \pmod{r}), \\
   r-1 & (n \equiv 0 \pmod{r}).
  \end{cases}
\label{eqn:Gauss sum trivial psi}
\end{eqnarray}
 
Let $\mathbb{P}_N$ be a set of odd prime numbers not dividing $N$ such that, for any positive integers $a, b$ coprime to each other, $\mathbb{P}_N$ contains a prime number $r$ of the form $r=am+b$ for some $m \in \Z_{>0}$.   
For an $r \in \mathbb{P}_N$, denote by $X_r$ the set of all Dirichlet characters mod $r$ (including the principal character $\psi_{r,0}$). 
For $\psi \in X_r$, we define the Dirichlet character  $\psi^*$  by
\begin{equation}
\psi^*(k)=\overline{\psi(k)}\left(\frac{k}{r}\right)^{\ell}. 
\label{eqn:psi-star}
\end{equation}
We put 
\begin{equation}
 C_{\ell,r} =
  \begin{cases}
  1 & (\text{$\ell$ is even}), \\
  \varepsilon_{r}^{\ell}  &   (\text{$\ell$ is odd}).
 \end{cases}
\label{eqn:def of clr}
\end{equation}
(For the definition of $\varepsilon_r$, see \eqref{eqn:def of epsilond}.)

In the following, we fix a Dirichlet character $\chi$  mod ${N}$
 that satisfies $\chi(-1)=i^{\ell}$ (resp.\ $\chi(-1)=1$)
 when $\ell$ is even (resp.\ odd).

For an $r \in \mathbb{P}_N$ and a $\psi \in X_r$, we consider the following conditions $\mathrm{[A1]}_{r,\psi}$ -- $\mathrm{[A5]}_{r,\psi}$ on  $\xi_{\pm}(\alpha,\psi;s)$ and $\xi_{\pm}(\beta,\psi^*;s)$.  
 
\begin{description}
\item[$\text{\bf [A1]}_{r,\psi}$] 
 $\xi_{\pm}(\alpha,\psi;s), \xi_{\pm}(\beta,\psi^*;s)$  have meromorphic 
 continuations to the whole $s$-plane, and $(s-1)(s-2+2\lambda)\xi_{\pm}(\alpha,\psi;s), (s-1)(s-2+2\lambda)\xi_{\pm}(\beta,\psi^*;s)$ are entire functions, which are of finite order in any vertical strip. 
\item[$\text{\bf [A2]}_{r,\psi}$] 
The residues of  $\xi_{\pm}(\alpha,\psi;s)$ and $\xi_{\pm}(\beta,\psi^*;s)$ satisfy
\[
\mathop{\mathrm{Res}}_{s=1} \xi_+(\alpha,\psi;s) = \mathop{\mathrm{Res}}_{s=1} \xi_-(\alpha,\psi;s), \quad 
\mathop{\mathrm{Res}}_{s=1} \xi_+(\beta,\psi^*;s) = \mathop{\mathrm{Res}}_{s=1} \xi_-(\beta,\psi^*;s). 
\]
\item[$\text{\bf [A3]}_{r,\psi}$] 
$\Xi_{\pm}(\alpha,\psi;s)$ and $\Xi_{\pm}(\beta,\psi^*;s)$
satisfy the following functional equation:
\begin{align*}
 \gamma(s)\left(
 \begin{array}{c}
 \Xi_{+}(\alpha,\psi;s) \\ \Xi_{-}(\alpha,\psi;s)
 \end{array}
 \right) &= 
 \chi(r) \cdot C_{\ell,r} \cdot  \psi^*(-N) \cdot 
 r^{2\lambda-2}\cdot (Nr^2)^{2-2\lambda-s} \\
 &\qquad \cdot \Sigma(\ell)\cdot  \gamma(2-2\lambda-s)
 \begin{pmatrix}
 \Xi_{+}\left(\beta, {\psi}^*;2-2\lambda-s\right) \\[4pt]
 \Xi_{-}\left(\beta, {\psi}^*;2-2\lambda-s\right)
 \end{pmatrix}, 
 \end{align*}
where $\gamma(s)$ and $\Sigma(\ell)$ are the same as
\eqref{form:DefOfGammaAndSigma} in {\bf [A3]}.
%
\item[$\text{\bf [A4]}_{r,\psi}$] 
If $\lambda = \frac q2$ $(q \in \Z_{\geq0},\ q \geq 4)$, then
\[
\xi_+(\alpha,\psi;-k)+(-1)^k\xi_-(\alpha,\psi;-k)=0 \quad (k = 1,2,\ldots,q-3).
\]

\item[$\text{\bf [A5]}_{r,\psi}$] 
The following four relations between residues and special values hold: 
\end{description}
$$
\leqno{\bullet} \quad 
 \xi_e(\alpha,\psi;0) 
  = \tau_\psi(0)
          \xi_e(\alpha;0). 
$$
$$
\leqno{\bullet} \quad
 {\chi(r)}\cdot \psi^*(-N)\cdot C_{\ell,r}\cdot r^{2\lambda}
 {\displaystyle \Res_{s=1}}\,\xi_{\pm}(\beta, \psi^{*};s) 
     =\tau_{\psi}(0){\displaystyle \Res_{s=1}}\,\xi_{\pm}(\beta;s).
$$
$$
\leqno{\bullet} \quad 
 \xi_e(\beta, \psi^{*};0)
   = \tau_{\psi^*}(0) \xi_e(\beta;0).
$$
$$
\leqno{\bullet} \quad
 {\displaystyle \Res_{s=1}}\,\xi_{\pm}(\alpha,\psi;s) 
  =  {\chi(r)}\cdot \psi^*(-N)\cdot C_{\ell,r}\cdot r^{-2\lambda}
 \cdot \tau_{\psi^*}(0)\cdot {\displaystyle \Res_{s=1}}\,\xi_{\pm}(\alpha;s).
$$

\bigskip

\begin{lemma}[Converse Theorem]
 \label{corollary:Maassforms}
We assume that $\xi_\pm(\alpha;s)$ and $\xi_\pm(\beta;s)$ satisfy the  conditions
 {\rm [A1] -- [A4]}, and define $\alpha(0), \alpha(\infty), \beta(0), \beta(\infty)$ by 
$(\ref{form:DefOfA0AndAInfty})$, $(\ref{form:DefOfA0AndAInfty2})$, 
$(\ref{form:DefOfB0AndBInfty})$, $(\ref{form:DefOfB0AndBInfty2})$, 
respectively.   
We assume furthermore that, for any $r \in \mathbb{P}_N$ and $\psi \in X_r$, 
 $\xi_{\pm}(\alpha,\psi;s)$ and $\xi_{\pm}(\beta,\psi^*;s)$  satisfy the conditions $\mathrm{[A1]}_{r,\psi}$ -- $\mathrm{[A5]}_{r,\psi}$.
Define the functions $F_\alpha(z)$ and  $G_\beta(z)$ on the upper half plane $\mathcal{H}$ by 
\begin{align}
 \nonumber
 F_\alpha(z) &= \alpha(\infty) \cdot y^{\lambda-\ell/4}+
 \alpha(0)\cdot i^{-\ell/2} \cdot 
 \frac{(2\pi) 2^{1-2\lambda} \Gamma(2\lambda-1)}
 {\Gamma\left(\lambda+\frac{\ell}{4}\right)
 \Gamma\left(\lambda-\frac{\ell}{4}\right)} \cdot y^{1-\lambda-\ell/4} \\[3pt]
 \label{form:Results}
  &\quad + \sum
 \begin{Sb}
 n=-\infty \\
 n\neq 0
 \end{Sb}^{\infty} \alpha(n) \cdot 
 \frac{i^{-\ell/2}\cdot \pi^{\lambda} \cdot  |n|^{\lambda-1}}
{\Gamma\left(\lambda+\frac{\sgn(n)\ell}{4}\right)} \cdot 
 y^{-\ell/4}\, W_{\frac{\sgn(n)\ell}{4}, \lambda-\frac{1}{2}}\left(4\pi|n|y\right)
\cdot \mathbf{e}[nx],  \\
 \nonumber
G_\beta(z) &=  N^\lambda \beta(\infty) \cdot y^{\lambda-\ell/4} 
  + N^{1-\lambda}  \beta(0) \cdot i^{-\ell/2} \cdot
 \frac{(2\pi) 2^{1-2\lambda} \Gamma(2\lambda-1)}
 {\Gamma\left(\lambda+\frac{\ell}{4}\right)
 \Gamma\left(\lambda-\frac{\ell}{4}\right)} \cdot y^{1-\lambda-\ell/4} \\[3pt]
 \label{form:Results2}
  &\quad + N^{1-\lambda}  \sum
 \begin{Sb}
 n=-\infty \\
 n\neq 0
 \end{Sb}^{\infty} \beta(n) \cdot 
 \frac{i^{-\ell/2} \cdot \pi^{\lambda} \cdot  |n|^{\lambda-1}}
{\Gamma\left(\lambda+\frac{\sgn(n)\ell}{4}\right)} \cdot 
 y^{-\ell/4}\, W_{\frac{\sgn(n)\ell}{4}, \lambda-\frac{1}{2}}\left(4\pi|n|y\right)
\cdot \mathbf{e}[nx]. 
\end{align}
Here $W_{\mu, \nu}(x)$ denotes the Whittaker function.
Then $F_\alpha(z)$ (resp.\ $G_\beta(z)$) gives a Maass form 
 for ${\Gamma}_0(N)$ of weight $\frac{\ell}{2}$ with
character $\chi$ (resp.\ $\chi_{N,\ell}$), and eigenvalue $(\lambda-\ell/4)(1-\lambda-\ell/4)$, where
\begin{equation}
\chi_{N,\ell}(d)=\overline{\chi(d)}\left(\frac{N}{d}\right)^\ell.
\label{eqn:chi N ell}
\end{equation}
Moreover, 
we have 
\[
F_{\alpha}\left(-\frac{1}{Nz}\right) 
(\sqrt{N} z)^{-\ell/2} = G_{\beta}(z).
\]
\end{lemma}

\begin{remark}
In \cite{MSSU}, we proved the converse theorem 
under a weaker condition
$\lambda\not\in \frac12 -\frac12 \Z_{\geq 0}$. 
In the present paper, 
since the case of $\lambda=1$ will not be treated,
we assume 
$\lambda\not\in 1-\frac{1}{2}\Z_{\geq 0}$, 
which simplifies the description of
the converse theorem. 
\end{remark}

\section{Prehomogeneous vector spaces}
\label{section:PV}

Let $Y$ be a non-degenerate half-integral symmetric matrix
of degree $m$, and let $p$ be the number of 
positive eigenvalues of $Y$. 
{\it Throughout the present paper, we assume that 
$m\geq 5$ and $p(m-p)>0$. }
We denote by $SO(Y)$ the special orthogonal group
of $Y$ defined by
$SO(Y)=\{g\in SL_m(\C)\; |\; {}^{t}g Y g=Y\}$.
We define the representation $\rho$ of 
$G=GL_1(\C)\times SO(Y)$ on $V=\C^{m}$ by
\[
 \rho(\tilde{g})v=\rho(t, g)v  = tg v \qquad
 (\tilde{g}=(t, g)\in G, v\in V).
\]
Let $P(v)$ be the quadratic form on $V$ 
defined by
\begin{equation}
\label{form:DefOfRelInv}
P(v)=Y[v]={}^{t} v Y v,  
\end{equation}
where we use Siegel's notation. 
Then, for $\tilde{g}=(t, g)\in G$ and  $v\in V$, we have
\begin{equation}
\label{form:RelInvPv}
P(\rho(\tilde{g})v)= \chi(t, g)P(v), \quad 
\text{with} \quad
\chi(t, g)=t^2, 
\end{equation}
and $V-S$ is a single $\rho(G)$-orbit, where 
$S$ is the zero set of $P$:
\[
S=\{v\in V\; |\; P(v)=0\}.
\]
That is, $(G, \rho, V)$ is a reductive regular prehomogeneous vector space. 
(We refer to \cite{PVBook, SatoShintani} for the basics of the theory of 
prehomogeneous vector spaces.)
We identify the dual space $V^*$ of $V$ with $V$ itself via
the inner product $\langle v, v^*\rangle = {}^{t}v v^*$.
Then the dual triplet 
$(G, \rho^*, V^*)$ is given by
\[
 \rho^*(\tilde{g})v^*=\rho^*(t, g)v^*  = t^{-1} \cdot {}^{t}g^{-1} v^* \qquad
 (\tilde{g}=(t, g)\in G, v\in V^*).
\]
We define the quadratic form $P^*(v^*)$ on $V^*$ by
\begin{equation}
\label{form:DefOfPstar}
P^*(v^*)= \frac{1}{4} Y^{-1}[v^*] = \frac14 \cdot {}^{t}v^* \, Y^{-1} v^{*}.
\end{equation}
Then, for 
$\tilde{g}=(t, g)\in G$, $v^*\in V^*$, we have
\begin{equation}
P^*(\rho^*(\tilde{g})v^*)= \chi^*(t, g)P^*(v^*), \quad 
\text{with} \quad
\chi^*(t, g)=t^{-2}, 
\end{equation}
and $V-S^*$ is a single $\rho^*({G})$-oribit, where
$S^*$ is the zero set of $P^*$:
\[
S^{*}=\{v^*\in V^*\, |\, P^{*}(v^*)=0\}.
\]
%
%
For $\epsilon, \eta = \pm$, we put
\[
V_{\epsilon}=\{v\in V_{\R}\, |\, \sgn P(v) = \epsilon\}, \qquad 
V_{\eta}^*=\{v^*\in V_{\R}\, |\, \sgn P^*(v^*) = \eta\}.
\]
We denote by $dv= dv_1\cdots dv_m$ 
the Lebesgue measure on $V_{\R}$, and by
$\Sch{V_{\R}}$ the space of rapidly decreasing functions on $V_{\R}$.
Then, for $f, f^*\in \mathcal{S}(V_{\R})$ and
$\epsilon, \eta = \pm$, we define the local zeta functions
$\Phi_{\epsilon}(f; s)$ and 
$\Phi^*_{\eta}(f^*;s)$ by
\begin{equation}
\label{form:DefOfLocalZeta}
\Phi_{\epsilon}(f; s) = \int_{V_{\epsilon}} f(v)
|P(v)|^{s-\frac{m}{2}} dv, \quad 
\Phi_{\eta}^*(f^*; s) = \int_{V_{\eta}^*} f^*(v^*)
|P^*(v^*)|^{s-\frac{m}{2}} dv^{*}.
\end{equation}
For $\Re(s)>\frac{m}{2}$, the integrals
$\Phi_{\epsilon}(f; s)$ and 
$\Phi^*_{\eta}(f^*;s)$ converge absolutely,  
and as functions of $s$, 
they can be continued analytically 
to the whole $s$-plane as meromorphic functions. 
Further, we define the Fourier transform 
$\widehat{f}(v^*)$ of $f\in \mathcal{S}(V_{\R})$
by
\[
\widehat{f}(v^*) =\int_{V_{\R}}
f(v) \bfe[\langle v, v^*\rangle] dv.
\]
The following lemma is due to 
Gelfand-Shilov~\cite{GS};
a detailed proof is given in 
Kimura~\cite[\S~4.2]{PVBook}.

\begin{lemma}[Local Functional Equation]
\label{lemma:LFEfromGS}
Let $p$ be the number of positive eigenvalues of $Y$, and 
put $D=\det(2Y)$. Then the following functional equation holds: 
\begin{multline*}
\begin{pmatrix}
\Phi_{+}^*(\widehat{f}\, ; s) \\[4pt]
\Phi_{-}^*(\widehat{f}\, ; s) 
\end{pmatrix}=
\Gamma\left(s+1-\frac{m}{2}\right)\Gamma(s) |D|^{\frac{1}{2}}
\cdot 2^{-2s+\frac{m}{2}}\cdot \pi^{-2s+\frac{m}{2}-1} \\
\times
\begin{pmatrix}
\sin \pi\left(\frac{p}{2}-s\right) & 
\sin \frac{\pi p}{2} \\[4pt]
\sin \frac{\pi(m-p)}{2} & 
\sin \pi\left(\frac{m-p}{2}-s\right)
\end{pmatrix}
\begin{pmatrix}
\Phi_{+}\left(f; \frac{m}{2}-s\right)\\[4pt]
\Phi_{-}\left(f; \frac{m}{2}-s\right)
\end{pmatrix}.
\end{multline*}
\end{lemma}

In the rest of this section, we  investigate 
singular distributions whose supports are contained in
the real points $S_{\R}$ of $S$;
these distributions play an important role in
the calculation of residues of Siegel's zeta functions. 
We decompose $S_{\R}$ as
\[
 S_{\R}= S_{1, \R}\cup S_{2, \R}, \quad 
 S_{1, \R}=\{v\in V_{\R}\; |\; P(v)=0, v\neq 0\}, \quad
 S_{2,\R}=\{0\}.
\]
A measure on $S_{\R}$ that is $SO(Y)_{\R}$-invariant is 
constructed as follows.
Since $P(v)$ is a non-degenerate quadratic forms, we have
\[
 S_{1, \R}= \bigcup_{i=1}^{m} U_{i}, \qquad
 U_{i}= \left\{ v\in S_{1, \R}\, \bigg|\, 
 \frac{\partial}{\partial v_{i}}P(v)\neq 0\right\}.
\]
For $i=1,\dots, m$, we define
an $(m-1)$-dimensional differential form $\omega_i$ on $U_i$ by
\begin{equation}
 \label{form:DefOfOmegaI}
 \omega_i= (-1)^{i-1}
 \left(\frac{\partial}{\partial v_i} P(v)\right)^{-1} dv_1\wedge \cdots\wedge dv_{i-1}
 \wedge dv_{i+1}\wedge \cdots \wedge dv_{m}.
\end{equation}
It is easy to see that
there exists an $(m-1)$-dimensional differential form 
$\omega$ on $S_{\R}$ that satisfies
\[
 \omega|_{U_i} = \omega_i \qquad (i=1,\dots, m)
\]
and
\[
 dP(v) \wedge \omega = dv \; (=dv_1\wedge \cdots \wedge dv_m).
\]
Since $P(gv)=P(v)$ for $g\in SO(Y)_{\R}$, we have
\[
dv= dP(v)\wedge \omega(v)= dP(v)\wedge \omega(gv).
\]
Further, $\omega(tv)= t^{m-2}\omega(v)$ for $t>0$. 
Now let $|\omega(v)|_{\infty}$ denote the measure on 
$S_{1, \R}$ defined by $\omega$. Then we have
\begin{equation}
 \label{form:InvarianceOmega}
 |\omega(\rho(t,g)v)|_{\infty} = |\chi(t, g)|^{\frac{m}{2}-1} \cdot 
 |\omega(v)|_{\infty}
\end{equation}
for $g\in SO(Y)_{\R}$ and $t>0$.
Similary, for the zero set $S^*$ of $P^*$, we decompose 
the real points $S_{\R}^*$ as 
\[
 S_{\R}^{*}= S_{1, \R}^{*}\cup S_{2, \R}^{*}, \qquad 
 S_{1, \R}^*=\{v^*\in V_{\R}\; |\; P^*(v^*)=0, v^*\neq 0\}, \quad
 S_{2,\R}^*=\{0\}.
\]
The same argument as above ensures the existence of
an $(m-1)$-dimensional differential form $\omega^*$ on
$S_{1, \R}^*$ such that the restriction of $\omega^*$
on
\[
 U_i^* =\left\{ v^*\in S_{1, \R}^*\, \bigg|\, 
 \frac{\partial}{\partial v_{i}^*}P^*(v^*)\neq 0\right\}\qquad
 (i=1,\dots, m)
\]
is given by
\[
 \omega^*|_{U_i^*}= (-1)^{i-1}
 \left(\frac{\partial}{\partial v_i^*} P^*(v^*)\right)^{-1} 
 dv_1^*\wedge \cdots\wedge dv_{i-1}^*
 \wedge dv_{i+1}^*\wedge \cdots \wedge dv_{m}^{*}.
\]
We have
\begin{equation}
 \label{form:InvarianceOmegaStar}
 |\omega^*(\rho^*(t, g)v^*)|_{\infty}=|\chi(t, g)|^{1-\frac{m}{2}}
 \cdot |\omega^*(v^*)|_{\infty}
\end{equation}
for $g\in SO(Y)_{\R}$ and $t>0$, where 
$|\omega^*|_{\infty}$ denotes the measure on 
$S_{1, \R}^*$ defined by $\omega^*$. 
We refer to \cite[Chap.~III]{GS} for further details on 
the measures 
$|\omega|_{\infty}, |\omega^*|_{\infty}$.
Then we have the following 
\begin{lemma} 
\label{lemma:SingularInvariantDistributions}

\begin{enumerate}
\item 
If $f\in C_0^{\infty}(V_{\R}-S_{\R})$, then we have
\begin{multline*}
\int_{S_{1,\R}^*}\widehat{f}(v^*) |\omega^*(v^*)|_{\infty} 
=
\Gamma\left(\frac{m}{2}-1\right) |D|^{\frac{1}{2}}
\cdot 2^{2-\frac{m}{2}}\cdot \pi^{1-\frac{m}{2}} \\
\times
\begin{pmatrix}
\sin \dfrac{\pi}{2}(m-p) & 
\sin \dfrac{\pi p}{2} 
\end{pmatrix}
\begin{pmatrix}
\displaystyle
\int_{V_{+}} f(v) |P(v)|^{1-\frac{m}{2}} dv \\[5pt]
\displaystyle
\int_{V_{-}} f(v) |P(v)|^{1-\frac{m}{2}} dv \\
\end{pmatrix}.
\end{multline*}
\item 
If  $\widehat{f}\in C_0^{\infty}(V_{\R}-S_{\R}^*)$, then we have
\begin{multline*}
\int_{S_{1,\R}} f(v) |\omega(v)|_{\infty} 
=
\Gamma\left(\frac{m}{2}-1\right) |D|^{-\frac{1}{2}}
\cdot 2^{2-\frac{m}{2}}\cdot \pi^{1-\frac{m}{2}} \\
\times
\begin{pmatrix}
\sin \dfrac{\pi}{2}(m-p) & 
\sin \dfrac{\pi p}{2} 
\end{pmatrix}
\begin{pmatrix}
\displaystyle
\int_{V_{+}^*} \widehat{f}(v^*)|P^*(v^*)|^{1-\frac{m}{2}} dv^* \\[5pt]
\displaystyle
\int_{V_{-}^*} \widehat{f}(v^*)|P^*(v^*)|^{1-\frac{m}{2}} dv^* \\
\end{pmatrix}.
\end{multline*}
\end{enumerate}
\end{lemma}
This is stated, {\it without proof},  on p.\ 156 of  Sato-Shintani~\cite{SatoShintani} where
Siegel's zeta function is picked up as an example of their theory. 
Since the details cannot be found in other literature, 
we give a proof of the lemma 
for convenience of readers.

\begin{proof}
For $f\in C_0^{\infty}(V_{\R}-S_{\R})$, we consider the integral
\[
 \int_{S_{1,\R}^*}\widehat{f}(v^*) |\omega^*(v^*)|_{\infty}.
\]
We may replace $S_{1,\R}^*$  by
$S_{\R}^*=\{v^*\in V_{\R}\,|\, P^*(v^*)=0\}$, 
since $S_{2,\R}^*=\{0\}$ has measure $0$ in $S_{\R}^*$.
From the identity (19) (or the first formula on p.\ 257) in 
Gelfand-Shilov~\cite[Chap~III, \S 2.2]{GS}, we have
\[
 \int_{S_{\R}^*}\widehat{f}(v^*) |\omega^*(v^*)|_{\infty} =
\Res_{s=0} \int_{P^*(v^*)>0} P^*(v^*)^{s-1}\cdot \widehat{f}(v^*)dv^*.
\]
By the shift $s\mapsto s-\frac{m}{2}$, we have 
\begin{align*}
 \int_{S_{\R}^*}\widehat{f}(v^*) |\omega^*(v^*)|_{\infty} 
&=\Res_{s=\frac{m}{2}} \int_{V_{+}^*} P^{*}(v^*)^{s-\frac{m}{2}-1} \cdot \widehat{f}(v^*) dv^*\\
&= \lim_{s\to +\frac{m}{2}} \left(s-\frac{m}{2}\right) \Phi_+^{*}\left(\widehat{f}; s-1\right).
\end{align*}
It then follows from the local functional equation
(Lemma~\ref{lemma:LFEfromGS}) that 
\begin{align*}
\lim_{s\to +\frac{m}{2}} \left(s-\frac{m}{2}\right) \Phi_+^{*}\left(\widehat{f}; s-1\right)
&=
\lim_{s\to +\frac{m}{2}} \left(s-\frac{m}{2}\right)
\Gamma\left(s-\frac{m}{2}\right)\Gamma(s-1) |D|^{\frac{1}{2}}
\cdot 2^{-2s+\frac{m}{2}+2}\cdot \pi^{-2s+\frac{m}{2}+1} \\
&\qquad \times
\begin{pmatrix}
\sin \pi\left(\dfrac{p}{2}-s+1\right) & 
\sin \dfrac{\pi p}{2} 
\end{pmatrix}
\begin{pmatrix}
\displaystyle
\int_{V_{+}} f(v) |P(v)|^{1-s} dv \\[5pt]
\displaystyle
\int_{V_{-}} f(v) |P(v)|^{1-s} dv \\
\end{pmatrix} \\
&=\Gamma\left(\frac{m}{2}-1\right) |D|^{\frac{1}{2}}
\cdot 2^{2-\frac{m}{2}}\cdot \pi^{1-\frac{m}{2}} \\
&\qquad \times
\begin{pmatrix}
\sin \dfrac{\pi}{2}(m-p) & 
\sin \dfrac{\pi p}{2} 
\end{pmatrix}
\begin{pmatrix}
\displaystyle
\int_{V_{+}} f(v) |P(v)|^{1-\frac{m}{2}} dv \\[5pt]
\displaystyle
\int_{V_{-}} f(v) |P(v)|^{1-\frac{m}{2}} dv \\
\end{pmatrix},
\end{align*}
which proves 
the first assertion of 
Lemma~\ref{lemma:SingularInvariantDistributions}.
The second assertion can be proved in a similar fashion.
\end{proof}

\section{Siegel's zeta functions
with congruence conditions}
\label{section:SiegelZeta}

In this section, followig M.\ Sato-Shintani~\cite{SatoShintani}, 
we define Siegel's zeta functions associated with
$(G, \rho, V)$, and give their integral representations.
Moreover, we calculate the singular parts of the zeta integrals.  
For this calculation, we also refer to Kimura~\cite{PVBook}.
Furthermore, 
following F.\ Sato~\cite{SatoZetaDist}, 
we slightly generalize
Siegel's zeta functions   
with using Schwartz-Bruhat functions
on $\Q^m$ in order to treat the twisted zeta functions simultaneously.
Let $dx$ be the measure on $GL_{m}(\R)$ defined by
\[
 dx = |\det x|^{-m} \prod_{1\leq i,j\leq m} dx_{ij} \qquad 
 \text{for} \quad x = (x_{ij})\in GL_{m}(\R),
\]
and $d\lambda$ the measure on the space 
$\text{Sym}_{m}(\R)$ of symmetric matrices
of degree $m$ defined by 
\[
 d\lambda = |\det \lambda|^{-\frac{m+1}{2}} 
 \prod_{1\leq i\leq j\leq m} d\lambda_{ij} \qquad 
 \text{for} \quad \lambda = (\lambda_{ij})\in \text{Sym}_{m}(\R).
\]
Then we normalize a Haar measure $d^{1}g$ on 
the Lie group $SO(Y)_{\R}$ in such a way that 
the integration formula
\begin{equation}
\int_{GL_m(\R)} F(x)dx =
\int_{SO(Y)_{\R}\backslash GL_m(\R)} d\lambda({}^{t}\dot{x} Y \dot{x})
\int_{SO(Y)_{\R}}F(g\dot{x}) d^{1}g
\end{equation}
holds for all integrable functions $F(x)\in L^{1}(GL_{m}(\R))$.
Further, let $dt$ be the Lebesgue measure on $\R$ and 
put
\begin{equation}
 \label{form:DtimesT}
d^{\times}t =\dfrac{2 dt}{t}. 
\end{equation}
By \eqref{form:RelInvPv}, $|P(v)|^{-\frac{m}{2}} dv$ is an
$\R_{+}\times SO(Y)_{\R}$-invariant measure on $V_{\epsilon}$ and 
the isotropy subgroup
\[
SO(Y)_{v} =\{g\in SO(Y)\, |\, gv = v\}
\]
at $v\in V-S$ is a reductive algebraic group.
Hence, for $v\in V_{\epsilon}$, there exists a 
Haar measure $d\mu_{v}$ on $SO(Y)_{v, \R}$
such that the integration formula
\begin{multline}
\label{form:NormalizationGv}
 \int_0^{\infty} d^{\times}t\int_{SO(Y)_{\R}} H(t, g)  d^{1} g \\
 =
 \int_{0}^{\infty} \int_{SO(Y)_{\R}/SO(Y)_{v, \R}}
 |P(\rho(t, \dot{g})v)|^{-\frac{m}{2}} d(\rho(t, \dot{g})v) 
 \int_{SO(Y)_{v, \R}} H(t, \dot{g}h) d\mu_{v}(h)
\end{multline}
holds for all integrable functions $H(t, g)\in L^{1}(G_{\R})$.
Similarly, for $v^*\in V_{\eta}^*$, we write
\[
SO(Y)_{v^*, \R}=\{g\in SO(Y)_{\R}\, |\, {}^{t}g^{-1} v^*= v^*\}
\]
and fix a Haar measure $d\mu_{v^*}^*$ on $SO(Y)_{v^*, \R}$
such that the integration formula
\begin{multline}
\label{form:NormalizationGvStar}
 \int_0^{\infty} d^{\times}t \int_{SO(Y)_{\R}} H(t, g)  d^{1} g \\
 =
 \int_{0}^{\infty} \int_{SO(Y)_{\R}/SO(Y)_{v^*, \R}}
 |P^*(\rho^*(t, \dot{g})v^*)|^{-\frac{m}{2}} d(\rho^*(t, \dot{g})v^*) 
 \int_{SO(Y)_{v^*, \R}} H(t, \dot{g}h) d\mu_{v^*}^*(h)
\end{multline}
holds for all integrable functions
$H(t, g)\in L^{1}(G_{\R})$.

We call a function $\phi:V_{\Q}\to\C$ a
{\it Schwartz-Bruhat} function if the following two conditions are 
satisfied:
\begin{enumerate}
\item there exists a positive integer $M$ such that 
$\phi(v)=0$ for $v\not\in \frac{1}{M}V_{\Z}$, and 
\item there exists a positive integer $N$ such that
if $v, w\in V_{\Q}$
satisfy $v-w\in NV_{\Z}$. then
$\phi(v)=\phi(w)$.
\end{enumerate}
The totality of Schwartz-functions on $V_{\Q}$ is denoted
by $\Sch{V_{\Q}}$.
We define the Fourier transform
$\widehat{\phi}\in \mathcal{S}(V_{\Q})$ of a 
Schwartz-Bruhat function
$\phi\in \mathcal{S}(V_{\Q})$ by
\begin{equation}
\label{form:FourierTransformOverQ}
 \widehat{\phi}(v^*)=\frac{1}{[V_{\Z}: rV_{\Z}]}
 \sum_{v\in V_{\Q}/ r V_{\Z}} \phi(v) \bfe[-\langle v, v^*\rangle],
\end{equation}
where $r$ is a sufficiently large positive integer such that 
the value $\phi(v)\bfe[-\langle v, v^*\rangle]$
depends only on the residue class $v\bmod{r V_{\Z}}$.
Though $r$ is not unique, the value $\widehat{\phi}(v^*)$
does not depend on the choice  of $r$.
The following lemma is essentially an adelic version of Poisson summation formula. 
\begin{lemma}[Poisson summation formula]
For 
$\phi\in \mathcal{S}(V_{\Q})$ and $f\in \mathcal{S}(V_{\R})$, 
\[
 \sum_{v^*\in V_{\Q}} \widehat{\phi}(v^*) \widehat{f}(v^*)
 = \sum_{v\in V_{\Q}} \phi(v) f(v).
\]
\end{lemma}

For 
$\tilde{g}=(t, g)\in G_{\R}=\R^{\times}\times SO(Y)_{\R}$, we put
$f_{\tilde{g}}(v)=f(\rho(t, g)v)=f(tgv)$. Since
\[
 \widehat{f_{\tilde{g}}}(v^*)=\int_{V_{\R}} f(tgv) \bfe[\langle v, v^*\rangle]dv=
 |t|^{-m} \int_{V_{\R}} f(v) \bfe[\langle t^{-1} g^{-1}v, v^*\rangle] dv
 =|t|^{-m} \cdot \widehat{f}(\rho^*(t, g)v^*),
\]
we have the following

\begin{lemma}
\label{lemma:PoissonTwistRhoG}
For $\tilde{g}=(t, g)\in G_{\R}$, 
$\phi\in \mathcal{S}(V_{\Q})$, $f\in \mathcal{S}(V_{\R})$, 
\[
 |t|^{-m} \sum_{v^*\in V_{\Q}} \widehat{\phi}(v^*) \widehat{f}(\rho^*(t, g)v^*)
 = \sum_{v\in V_{\Q}} \phi(v) f(\rho(t, g)v).
\]
\end{lemma}

In the following, we assume that 
$\phi\in  \mathcal{S}(V_{\Q})$ is $SO(Y)_{\Z}$-invariant.
That is, $\phi$ is assumed to satisfy
\begin{equation}
 \label{form:PhiIsSOYZinv}
 \phi(\gamma v) =\phi(v) \qquad \text{for}\quad
 v\in V_{\Q}, \gamma\in SO(Y)_{\Z}.
\end{equation}
Then we define the zeta integral
$Z(f, \phi; s)$ by
\begin{equation}
\label{form:DefofZetaInt}
 Z(f, \phi; s)= \int_{0}^{\infty} d^{\times} t\int_{SO(Y)_{\R}/SO(Y)_{\Z}}
 |\chi(t, g)|^{s} \sum_{v\in V_{\Q}-S_{\Q}} \phi(v) f(\rho(t, g)v) d^{1}g.
\end{equation}
Since $V_{\Q}-S_{\Q}$ can be decomposed as
\[
 V_{\Q}-S_{\Q} = \bigcup_{\epsilon=\pm} \bigcup_{v\in SO(Y)_{\Z}\backslash
 V_{\epsilon}\cap V_{\Q}} 
 \bigcup_{\gamma\in SO(Y)_{\Z}/SO(Y)_{v, \Z}} \gamma v,
\]
we have, by a formal calculation, 
\begin{align*}
&Z(f,  \phi; s) \\
&=\sum_{\epsilon=\pm}\sum_{v\in SO(Y)_{\Z}\backslash
 V_{\epsilon}\cap V_{\Q}}\int_{0}^{\infty}  d^{\times} t
 \int_{SO(Y)_{\R}/SO(Y)_{\Z}} |\chi(t, g)|^{s}
 \sum_{\gamma\in SO(Y)_{\Z}/SO(Y)_{v, \Z}}
 \phi(\gamma v) f(\rho(t, g)\gamma v) d^{1}g \\
 &\overset{\eqref{form:PhiIsSOYZinv}}{=}
 \sum_{\epsilon=\pm}\sum_{v\in SO(Y)_{\Z}\backslash
 V_{\epsilon}\cap V_{\Q}} \phi(v) \int_{0}^{\infty} d^{\times} t
 \int_{SO(Y)_{\R}/SO(Y)_{v, \Z}} |\chi(t,g)|^{s}
  f(\rho(t, g)v) d^{1}g,
\end{align*}
and further, by applying \eqref{form:NormalizationGv} to 
\[
H(t, g)=|\chi(t,g)|^{s}\cdot f(\rho(t, g)v) =\frac{|P(\rho(t, g)v)|^{s}}{|P(v)|^s}\cdot 
f(\rho(t, g)v),
\]
we have
\begin{multline*}
Z(f, \phi; s) 
=\sum_{\epsilon=\pm}\sum_{v\in SO(Y)_{\Z}\backslash
 V_{\epsilon}\cap V_{\Q}} \phi(v)\\
 \times \int_{0}^{\infty}  d^{\times} t
 \int_{SO(Y)_{\R}/SO(Y)_{v, \R}}
 \frac{|P(\rho(t, \dot{g})v)|^{s}}{|P(v)|^s}\cdot 
f(\rho(t, \dot{g})v) |P(\rho(t, \dot{g})v|^{-\frac{m}{2}} d(\rho(t, \dot{g})v)\\
\times
\int_{SO(Y)_{v, \R}/SO_{v, \Z}} d\mu_{v}(h).
\end{multline*}
In the following, for $v\in V_{\Q}-S_{\Q}$, we put
\begin{equation}
\label{form:DefOfMuV}
\mu(v)=\int_{SO(Y)_{v, \R}/SO_{v, \Z}} d\mu_{v}(h).
\end{equation}
Since it is assumed that $m\geq 5$,
 the generic isotropy subgroup $SO(Y)_{v}$ is a semisimple algebraic group, 
 and thus we have $\mu(v)<+\infty$. (cf.\ \cite[p.\ 184]{PVBook}.)
We further put $\rho(t, \dot{g})v=x$ in the right hand side above. 
Then, since $\R_{+}\times SO(Y)_{\R}/SO(Y)_{v, \R}\cong V_{\epsilon}$, we have
\begin{equation}
\label{form:ZetaIntUnfolding}
Z(f, \phi; s) 
=\sum_{\epsilon=\pm}\left\{\sum_{v\in SO(Y)_{\Z}\backslash
 V_{\epsilon}\cap V_{\Q}} \frac{\phi(v)\mu(v)}{|P(v)|^s}\right\} 
 \int_{V_{\epsilon}}
 f(x) |P(x)|^{s-\frac{m}{2}} dx.
\end{equation}
The Dirichlet series
\[
\zeta_{\epsilon}(\phi; s)=
\sum_{v\in SO(Y)_{\Z}\backslash
 V_{\epsilon}\cap V_{\Q}} \frac{\phi(v)\mu(v)}{|P(v)|^s}
\]
converges absolutely for $\Re(s)> \frac{m}{2}$, 
as will be explained in Remark~\ref{remark:SiegelZeta} shortly.
Hence the interchange of summation and integration, which leads to
\eqref{form:ZetaIntUnfolding}, can be justified under this condition. 
Similary, for 
$f^*\in \mathcal{S}(V_{\R})$ and 
$\phi^*\in \mathcal{S}(V_{\Q})$ that satisfies
\begin{equation}
\label{form:PhiStarIsSOYZinv}
 \phi^*({}^{t}\gamma^{-1} v^*) =\phi^*(v^*) \qquad \text{for}\quad
 v^*\in V_{\Q}, \gamma\in SO(Y)_{\Z},
\end{equation}
we define the zeta ingegral $Z^*(f^*, \phi^*;s)$ by
\begin{equation}
\label{form:DefofDualZetaInt}
 Z^*(f^*, \phi^*; s)= 
 \int_{0}^{\infty} d^{\times} t\int_{SO(Y)_{\R}/SO(Y)_{\Z}}
 |\chi^*(t, g)|^{s} \sum_{v^*\in V_{\Q}-S_{\Q}^*} 
 \phi^*(v^*) f^*(\rho^*(t, g)v^*) d^{1}g.
\end{equation}
Furthermore, for $v^*\in V_{\Q}-S_{\Q}^*$, we put
\begin{equation}
\label{form:DefOfMustarVstar}
\mu^*(v^*)=\int_{SO(Y)_{v^*, \R}/SO_{v^*, \Z}} d\mu_{v^*}^*(h),
\end{equation}
where $d\mu_{v^*}^*$ is the  Haar measure on 
$SO(Y)_{v^*, \R}$ defined by 
\eqref{form:NormalizationGvStar}.

\begin{definition}[Siegel's zeta functions with congruence conditions]
Let $\epsilon, \eta=\pm$ and assume that 
$\phi, \phi^*\in \mathcal{S}(V_{\Q})$ satisfy 
\eqref{form:PhiIsSOYZinv}, \eqref{form:PhiStarIsSOYZinv}, respectively. 
Then we define 
$\zeta_{\epsilon}(\phi; s)$ and $\zeta_{\eta}^*(\phi^*; s)$ by
\begin{align}
\label{form:DefOfSiegelZeta}
\zeta_{\epsilon}(\phi; s) &=
\sum_{v\in SO(Y)_{\Z}\backslash
 V_{\epsilon}\cap V_{\Q}} \frac{\phi(v)\mu(v)}{|P(v)|^s} \\[5pt]
 \label{form:DefOfDualSiegelZeta}
\zeta_{\eta}^*(\phi^*; s) &=
\sum_{v^*\in SO(Y)_{\Z}\backslash
 V_{\eta}^*\cap V_{\Q}} \frac{\phi^*(v^*)\mu^*(v^*)}{|P^*(v^*)|^s}.
\end{align}
\end{definition}

We can summarize our argument as the following

\begin{lemma}[Integral representations of the zeta functions]
\label{lemma:IntegralRepnOfZeta}
Let $f, f^*\in \mathcal{S}(V_{\R})$ and assume that 
$\phi, \phi^*\in \mathcal{S}(V_{\Q})$ are $SO(Y)_{\Z}$-invariant. 
For $\Re(s)> \frac{m}{2}$, we have 
\begin{align*}
 Z(f, \phi; s)&= \sum_{\epsilon=\pm}
 \zeta_{\epsilon}(\phi;s) \Phi_{\epsilon}(f; s), \\
Z^*(f^*, \phi^*; s) &= \sum_{\eta=\pm}
 \zeta_{\eta}^*(\phi^*;s) \Phi_{\eta}^*(f^*; s).
\end{align*}
\end{lemma}

\begin{remark}
\label{remark:SiegelZeta}
\begin{enumerate}
\item 
The original Siegel's zeta functions are obtained by letting 
$\phi=\phi_0$, where $\phi_0$ is the characteristic function 
$\text{ch}_{V_{\Z}}$ of $V_{\Z}$.
To apply Weil-type converse theorems, we need to examine the case where
$\phi(v)= \psi(P(v)) \phi_0(v)$ with Dirichlet character $\psi$.
Since each $\phi(v)$ is a linear combination of characteristic functions 
of subsets of the form $a + NV_{\Z} \; (a\in V_{\Q}, N\in \Z_{\geq 1})$, 
we call  $\zeta_{\epsilon}(\phi; s)$, $\zeta_{\eta}^*(\phi^*; s)$
{\it Siegel's zeta functions with congruence conditions.}  
\item 
The absolute convergence of Siegel's zeta functions is not at all obvious, 
though Siegel wrote just 
``Die Konvergents der Reihe entnimmt man der Reduktiontheorie''.
A detailed proof of the convergence can be found in 
Tamagawa~\cite{Tamagawa}. It also follows from 
the general theory of prehomogeneous vector spaces
(Saito~\cite{Saito}, F.\ Sato~\cite{SatoConv}).
\item
We can write $\zeta_{\pm}(\phi;s)$ as 
\[
 \zeta_{\pm}(\phi;s) =
 \sum_{r\in \Q_{>0}} \frac{M(P, \phi; \pm r)}{r^s}
\]
with
\[
 M(P, \phi; \pm r):=
 \sum
 \begin{Sb}
 v\in SO(Y)_{\Z}\backslash V_{\pm}\cap
 \supp(\phi) \\
 P(v) =\pm r
 \end{Sb} \phi(v) \mu(v).
\]
Since $\phi(v)= 0$ for $v\not\in \frac{1}{L}V_{\Z}$
with some integer $L$, we see that 
that the sum in the definition of
$M(P, \phi;\pm r)$ is a finite sum
(cf.\ Kimura~\cite[p.184]{PVBook}).
In the case of $\phi= \phi_0$, we have $\supp(\phi_0)=V_{\Z}$ and
$P(v)\in \Z\setminus\{0\}$ for $v\in V_{\pm}\cap V_{\Z}$.
For $n=1,2,\dots$, we put
\begin{equation}
\label{form:DefOfMPn}
M(P;  \pm n) = 
\sum
\begin{Sb}
v\in SO(Y)_{\Z}\backslash V_{\pm}\cap V_{\Z} \\
P(v)= \pm n
\end{Sb} \mu(v).
\end{equation}
Siegel called $M(P; n)$ 
the measures of representation ({\it Darstellungsma\ss{}}). 
We have
\[
 \zeta_{\pm}(\phi_0;s) =
\sum_{n=1}^{\infty} \frac{M(P; \pm n)}{n^s}.
\]
\end{enumerate}
\end{remark}
To investigate analytic properties of the zeta integrals, 
we define measures on isotropy subgroups at singular points.
We fix an arbitrary point $v$ of $S_{1, \R}$.  
Recall that in the previous section, 
we have defined an $SO(Y)_{\R}$-invariant measure
$|\omega|_{\infty}$ on 
$S_{1, \R}\cong SO(Y)_{\R}/SO(Y)_{v, \R}$.
We can normalize a measure  $d\sigma_{v}$ on the isotropy subgroup 
$SO(Y)_{v, \R}$ in such a way that 
the integration formula
\begin{equation}
\label{form:NormalizationDnu}
\int_{SO(Y)_{\R}} \psi(g) d^{1} g=
\int_{SO(Y)_{\R}/SO(Y)_{v, \R}} |\omega(\dot{g}v)|_{\infty}
\int_{SO(Y)_{v, \R}} \psi(\dot{g}h) d\sigma_{v}(h)
\end{equation}
holds for all integrable functions $\psi(g)\in L^{1}(SO(Y)_{\R})$.
Similarly, for $v^*\in S_{1, \R}^*$, we take a measure 
$d\sigma_{v^*}^*$ on the isotropy subgroup $SO(Y)_{v^*, \R}$
such that the integration formula
\begin{equation}
\label{form:NormalizationDnuStar}
\int_{SO(Y)_{\R}} \psi(g) d^{1} g=
\int_{SO(Y)_{\R}/SO(Y)_{v^*, \R}} |\omega^*({}^{t}\dot{g}^{-1}v^*)|_{\infty}
\int_{SO(Y)_{v^*, \R}} \psi(\dot{g}h) d\sigma_{v^*}^*(h)
\end{equation}
holds  for all integrable functions $\psi(g)\in L^{1}(SO(Y)_{\R})$.
Now we put
\begin{align*}
 Z_{+}(f, \phi; s) &= \int_{1}^{\infty} d^{\times} t\int_{SO(Y)_{\R}/SO(Y)_{\Z}}
 |\chi(t, g)|^{s} \sum_{v\in V_{\Q}-S_{\Q}} \phi(v) f(\rho(t, g)v) d^{1}g, \\
 Z_{-}(f, \phi; s) &= \int_{0}^{1} d^{\times} t\int_{SO(Y)_{\R}/SO(Y)_{\Z}}
 |\chi(t, g)|^{s} \sum_{v\in V_{\Q}-S_{\Q}} \phi(v) f(\rho(t, g)v) d^{1}g, \\
 Z_{+}^*(f^*,  \phi^*; s) &= 
 \int_{0}^{1} d^{\times} t\int_{SO(Y)_{\R}/SO(Y)_{\Z}}
 |\chi^*(t, g)|^{s} \sum_{v^*\in V_{\Q}-S_{\Q}^*} 
 \phi^*(v^*) f^*(\rho^*(t, g)v^*) d^{1}g, \\
 Z_{-}^*(f^*, \phi^*; s) &= 
 \int_{1}^{\infty} d^{\times} t\int_{SO(Y)_{\R}/SO(Y)_{\Z}}
 |\chi^*(t, g)|^{s} \sum_{v^*\in V_{\Q}-S_{\Q}^*} 
 \phi^*(v^*) f^*(\rho^*(t, g)v^*) d^{1}g.
\end{align*}
It is obvious that 
\[
Z(f,\phi;s )= Z_{+}(f, \phi;s)+Z_{-}(f, \phi;s), \quad 
Z^*(f^*,\phi^*;s)=
Z_{+}^*(f^*,\phi^*;s)+
Z_{-}^*(f^*,\phi^*;s).
\]
The four integrals above converges absolutely for $\Re(s)>\frac{m}{2}$, and
further, two integrals $Z_{+}(f, \phi; s)$ and  $Z_{+}^*(f^*, \phi^*; s)$
are absolutely convergent for any $s\in \C$ and define entire functions of $s$. 
Let us calculate $Z_{-}(f, \phi;s)$ {\it formally} by using 
Lemma~\ref{lemma:PoissonTwistRhoG}, the Poisson
summation formula; the interchange of integral and summation will be justified later in
Remark~\ref{remark:Interchange}.
Since $\chi(t, g)= \chi^*(t, g)^{-1}= t^2$, it follows from 
Lemma~\ref{lemma:PoissonTwistRhoG} that 
\begin{align*}
 Z_{-}(f, \phi; s) 
 &= \int_{0}^{1} d^{\times} t\int_{SO(Y)_{\R}/SO(Y)_{\Z}}
 |\chi(t, g)|^{s} \\
 &\qquad \times
 \left\{
 |t|^{-m} 
 \sum_{v^*\in V_{\Q}} \widehat{\phi}(v^*) \widehat{f}(\rho^*(t, g)v^*)-
 \sum_{v\in S_{\Q}} \phi(v) f(\rho(t, g)v)
 \right\}d^{1}g\\
 &= \int_{0}^{1} d^{\times} t\int_{SO(Y)_{\R}/SO(Y)_{\Z}}
 |\chi^*(t, g)|^{\frac{m}{2}-s} 
 \sum_{v^*\in V_{\Q}-S_{\Q}^*} \widehat{\phi}(v^*) \widehat{f}(\rho^*(t, g)v^*)d^{1}g
 \\
& \qquad + 
 \int_{0}^{1} t^{2s-m} d^{\times} t\int_{SO(Y)_{\R}/SO(Y)_{\Z}}
  \sum_{v^*\in S_{\Q}^*} \widehat{\phi}(v^*) \widehat{f}(\rho^*(t, g)v^*)d^{1}g \\
 & \qquad - 
 \int_{0}^{1} t^{2s} d^{\times} t\int_{SO(Y)_{\R}/SO(Y)_{\Z}}
 \sum_{v\in S_{\Q}} \phi(v) f(\rho(t, g)v) d^{1}g.
\end{align*}
The first term of the most right hand side is 
\[
\int_{0}^{1} d^{\times} t\int_{SO(Y)_{\R}/SO(Y)_{\Z}}
 |\chi^*(t, g)|^{\frac{m}{2}-s} 
 \sum_{v^*\in V_{\Q}-S_{\Q}^*} \widehat{\phi}(v^*) \widehat{f}(\rho^*(t, g)v^*)d^{1}g
 =Z_{+}^{*}\left(\widehat{f}, \widehat{\phi}; \frac{m}{2}-s\right).
\]
Using \eqref{form:NormalizationDnu} and \eqref{form:NormalizationDnuStar}, 
we calculate the second and third terms 
following the method of
Sato-Shintani~\cite[Theorem~2]{SatoShintani}.
  Put
\[
S_{1, \Q}=\{v\in V_{\Q}\; |\; P(v)=0, v\neq 0\}, \qquad
S_{1, \Q}^*=\{v^*\in V_{\Q}\; |\; P^*(v^*)=0, v^*\neq 0\}
\]
By the interchange of summation and integration, the third term above becomes
\begin{align}
 \label{form:SingularIntegral1}
 &\qquad \int_{0}^{1} t^{2s} d^{\times} t\int_{SO(Y)_{\R}/SO(Y)_{\Z}}
 \sum_{v\in S_{\Q}} \phi(v) f(\rho(t, g)v) d^{1}g \\
 \nonumber
 &=\sum_{v\in SO(Y)_{\Z}\backslash S_{1, \Q}}\phi(v)
 \int_0^{1} t^{2s} d^{\times}t \int_{SO(Y)_{\R}/SO(Y)_{v, \Z}}
 f(\rho(t, g)v) d^{1}g \\
 \nonumber
 &\qquad + 
 \phi(0) f(0)\int_0^{1} t^{2s} d^{\times} t \int_{SO(Y)_{\R}/SO(Y)_{\Z}} d^{1}g.
\end{align}
By applying \eqref{form:NormalizationDnu} to
$\psi(g)=f(\rho(t, g)v)=f(tgv)$, we have
\begin{align*}
\int_{SO(Y)_{\R}/SO(Y)_{v, \Z}}
 f(\rho(t, g)v) d^{1}g &= 
 \int_{SO(Y)_{\R}/SO(Y)_{v, \R}}  |\omega(\dot{g}v)|_{\infty}
 \int_{SO(Y)_{v, \R}/SO(Y)_{v, \Z}} f(t\dot{g} hv) d\sigma_{v}(h) \\
 &=\int_{S_{1,\R}}  f(tz) |\omega(z)|_{\infty}
 \int_{SO(Y)_{v, \R}/SO(Y)_{v, \Z}}  d\sigma_{v}(h) \\
 &=t^{2-m} 
 \int_{S_{1, \R}}  f(z) |\omega(z)|_{\infty}
 \int_{SO(Y)_{v, \R}/SO(Y)_{v, \Z}}  d\sigma_{v}(h).  
\end{align*}
Here we have used  \eqref{form:InvarianceOmega}
in the third equality. 
Hence 
the integral~\eqref{form:SingularIntegral1} is calculated as 
\begin{align*}
 &\qquad \int_{0}^{1} t^{2s} d^{\times} t\int_{SO(Y)_{\R}/SO(Y)_{\Z}}
 \sum_{v\in S_{\Q}} \phi(v) f(\rho(t, g)v) d^{1}g \\
 &=\sum_{v\in SO(Y)_{\Z}\backslash S_{1, \Q}}\phi(v)
 \int_0^{1} t^{2s+2-m} \, \frac{2dt}{t} 
 \int_{S_{1,\R}}  f(z) |\omega(z)|_{\infty}
 \int_{SO(Y)_{v, \R}/SO(Y)_{v, \Z}}  d\sigma_{v}(h) \\
 &\qquad + 
 \phi(0) f(0) \int_0^{1} t^{2s} \frac{2 dt}{t}
  \int_{SO(Y)_{\R}/SO(Y)_{\Z}} d^{1}g \\
  &=\frac{1}{s+1-\frac{m}{2}} \int_{S_{1,\R}}  f(z) |\omega(z)|_{\infty}
  \sum_{v\in SO(Y)_{\Z}\backslash S_{1, \Q}}\phi(v)
   \int_{SO(Y)_{v, \R}/SO(Y)_{v, \Z}}  d\sigma_{v}(h)\\
   &\qquad + \frac{\phi(0) f(0)}{s} \int_{SO(Y)_{\R}/SO(Y)_{\Z}} d^{1}g.
\end{align*}
Similarly, by term-by-term integration, we have
\begin{align}
 \label{form:SingularIntegral2}
  &\qquad
 \int_{0}^{1} t^{2s-m} d^{\times} t\int_{SO(Y)_{\R}/SO(Y)_{\Z}}
  \sum_{v^*\in S_{\Q}^*} \widehat{\phi}(v^*) \widehat{f}(\rho^*(t, g)v^*)d^{1}g \\
  \nonumber
  &=\sum_{v^*\in SO(Y)_{\Z}\backslash S_{1, \Q}^{*}}
  \widehat{\phi}(v^*) \int_0^1 t^{2s-m} d^{\times}t \int_{SO(Y)_{\R}/SO(Y)_{v^*, \Z}}
 \widehat{f}(\rho^*(t, g)v^*)d^{1}g \\
 \nonumber
 &\qquad + \widehat{\phi}(0) \widehat{f}(0) 
 \int_0^1 t^{2s-m} d^{\times} t \int_{SO(Y)_{\R}/SO(Y)_{\Z}} d^{1}g,
\end{align}
and by using \eqref{form:InvarianceOmegaStar}
and \eqref{form:NormalizationDnuStar}, we obtain
\begin{align*}
&\qquad \int_{SO(Y)_{\R}/SO(Y)_{v^*, \Z}} \widehat{f}(\rho^*(t, g)v^*)d^{1}g \\
&= \int_{SO(Y)_{\R}/SO(Y)_{v^*,\R}} 
\widehat{f}(t^{-1}\cdot 
 {}^{t}\dot{g}^{-1}v^*) |\omega^*({}^{t}\dot{g}^{-1}v^*)|_{\infty}
\int_{SO(Y)_{v^*, \R}/SO(Y)_{v^*, \Z}} d\sigma_{v^*}^{*}(h) \\
&= \int_{S_{1,\R}^{*}} \widehat{f}(t^{-1} z^*) |\omega^*(z^*)|_{\infty}
\int_{SO(Y)_{v^*, \R}/SO(Y)_{v^*, \Z}} d\sigma_{v^*}^{*}(h) \\
&=t^{m-2}\int_{S_{1,\R}^{*}} \widehat{f}(z^*) |\omega^*(z^*)|_{\infty}
\int_{SO(Y)_{v^*, \R}/SO(Y)_{v^*, \Z}} d\sigma_{v^*}^{*}(h).
\end{align*}
Hence we see that 
\begin{align*}
  &\qquad
 \int_{0}^{1} t^{2s-m} d^{\times} t\int_{SO(Y)_{\R}/SO(Y)_{\Z}}
  \sum_{v^*\in S_{\Q}^*} \widehat{\phi}(v^*) \widehat{f}(\rho^*(t, g)v^*)d^{1}g \\
  &=
  \frac{1}{s-1}\int_{S_{1,\R}^{*}} \widehat{f}(z^*) |\omega^*(z^*)|_{\infty}
  \sum_{v^*\in SO(Y)_{\Z}\backslash S_{1, \Q}^{*}}
  \widehat{\phi}(v^*) \int_{SO(Y)_{v^*, \R}/SO(Y)_{v^*, \Z}} d\sigma_{v^*}^{*}(h) \\
  &\qquad +\frac{\widehat{\phi}(0) \widehat{f}(0)}{s-\frac{m}{2}}
  \int_{SO(Y)_{\R}/SO(Y)_{\Z}} d^{1}g
\end{align*}
Now we put
\begin{align}
\label{form:DefOfNV}
\sigma(v) &:=
\int_{SO(Y)_{v, \R}/SO(Y)_{v, \Z}}  d\sigma_{v}(h), \\
\label{form:DefOfNVStar}
\sigma^*(v^*) &:=
\int_{SO(Y)_{v^*, \R}/SO(Y)_{v^*, \Z}} d\sigma_{v^*}^{*}(h).
\end{align}
Then we have the first assertion of the following 
lemma; the second assertion can be proved similarly as the first assertion, 
and then the third assertion follows immediately from the first and second
assertions.  
\begin{lemma}
\label{lemma:ACandFEofZetaIntegral}
\begin{enumerate}
\item 
 For $\Re(s)>\frac{m}{2}$, we have
\begin{align*}
 Z(f, \phi; s) &= Z_{+}(f, \phi; s)+
 Z_{+}^{*}\left(\widehat{f}, \widehat{\phi}; \frac{m}{2}-s\right)\\
 &\qquad +
   \frac{1}{s-1}\int_{S_{1,\R}^{*}} \widehat{f}(z^*) |\omega^*(z^*)|_{\infty}
  \sum_{v^*\in SO(Y)_{\Z}\backslash S_{1, \Q}^{*}}
  \widehat{\phi}(v^*) \sigma^*(v^*) \\
  &\qquad +\frac{\widehat{\phi}(0) \widehat{f}(0)}{s-\frac{m}{2}}
  \int_{SO(Y)_{\R}/SO(Y)_{\Z}} d^{1}g \\
 &\qquad -\frac{1}{s+1-\frac{m}{2}} \int_{S_{1,\R}}  f(z) |\omega(z)|_{\infty}
  \sum_{v\in SO(Y)_{\Z}\backslash S_{1, \Q}}\phi(v)
  \sigma(v) \\
   &\qquad - \frac{\phi(0) f(0)}{s} \int_{SO(Y)_{\R}/SO(Y)_{\Z}} d^{1}g.
\end{align*}
\item For $\Re(s)>\frac{m}{2}$, we have
\begin{align*}
 Z^*(\widehat{f}, \widehat{\phi}; s) &= Z_{+}^*(\widehat{f}, \widehat{\phi}; s)
+Z_{+}\left(f, \phi, \frac{m}{2}-s\right)\\
  &\qquad +\frac{1}{s-1} \int_{S_{1,\R}}  f(z) |\omega(z)|_{\infty}
  \sum_{v\in SO(Y)_{\Z}\backslash S_{1, \Q}}\phi(v)
   \sigma(v) \\
   &\qquad + \frac{\phi(0) f(0)}{s-\frac{m}{2}} \int_{SO(Y)_{\R}/SO(Y)_{\Z}} d^{1}g\\
  &\qquad -
  \frac{1}{s+1-\frac{m}{2}}\int_{S_{1,\R}^{*}} \widehat{f}(z^*) |\omega^*(z^*)|_{\infty}
  \sum_{v^*\in SO(Y)_{\Z}\backslash S_{1, \Q}^{*}}
  \widehat{\phi}(v^*) 
   \sigma^*(v^*) \\
  &\qquad -\frac{\widehat{\phi}(0) \widehat{f}(0)}{s}
  \int_{SO(Y)_{\R}/SO(Y)_{\Z}} d^{1}g.
\end{align*}
\item As functions of $s$, the integrals $Z(f, \phi, s)$ and 
 $Z^*(\widehat{f}, \widehat{\phi}; s)$ can be continued analytically 
 to the whole $s$-plane, and satisfy the following functional equation:
\[
  Z^*(\widehat{f}, \widehat{\phi}; s) = Z\left(f, \phi; \frac{m}{2}-s\right).
\]
\end{enumerate}
\end{lemma}

\begin{remark}
\label{remark:Interchange}
In~\cite{Igusa1971}, Igusa studied the so-called 
admissible representations related to 
the Siegel-Weil formula~\cite{WeilActa1965}. 
According to his classification, 
our prehomogeneous vector space 
$(GL_1(\C)\times SO(Y), \C^m)$ gives an 
admissible representation if $m\geq 5$, and this implies that the integrals
\[
\int_{SO(Y)_{\R}/SO(Y)_{\Z}}
  \sum_{v\in V_{\Q}} \phi(v) f(gv) d^{1}g, \qquad 
\int_{SO(Y)_{\R}/SO(Y)_{\Z}}
  \sum_{v^*\in V_{\Q}} \phi^*(v^*) f^*({}^{t}g^{-1}v^*) d^{1}g
\]
are absolutely convergent for all
Schwartz-Bruhat functions $f, f^*\in \Sch{V_{\R}}$ and  
$\phi, \phi^*\in \Sch{V_{\Q}}$.
Hence the integrals
\begin{align*}
\int_{SO(Y)_{\R}/SO(Y)_{\Z}}
  \sum_{v\in S_{1, \Q}} \phi(v) f(gv) d^{1}g &= 
 \int_{S_{1,\R}}  f(z) |\omega(z)|_{\infty}
  \sum_{v\in SO(Y)_{\Z}\backslash S_{1, \Q}}\phi(v)
   \sigma(v), \\
\int_{SO(Y)_{\R}/SO(Y)_{\Z}}
  \sum_{v^*\in S_{1, \Q}^*} \widehat{\phi}(v^*) \widehat{f}({}^{t}g^{-1}v^*) d^{1}g &=
\int_{S_{1,\R}^{*}} \widehat{f}(z^*) |\omega^*(z^*)|_{\infty}
  \sum_{v^*\in SO(Y)_{\Z}\backslash S_{1, \Q}^{*}}
  \widehat{\phi}(v^*) 
   \sigma^*(v^*),
\end{align*}
which appear in Lemma~\ref{lemma:ACandFEofZetaIntegral}, are 
absolutely convergent, 
and the interchange of integral and summation can be justified by
Fubini's theorem. 
\end{remark}

\section{Analytic properties of Siegel's zeta functions}
\label{section:PropertiesOfSiegelZeta}

\begin{theorem}
\label{theorem:SiegelZetaProperties}
Assume that
$\phi\in\mathcal{S}(V_{\Q})$ is $SO(Y)_{\Z}$-invariant.
\begin{enumerate}
\item The zeta functions
$\zeta_{\epsilon}(\phi; s)$ and $\zeta_{\eta}^*(\widehat{\phi};s)$
have analytic continuations of $s$ in $\C$, and
the zeta functions multiplied by $(s-1)(s-\frac{m}{2})$ are 
entire functions of $s$ of finite order in any vertical strip.
\item Th zeta functions
$\zeta_{\epsilon}(\phi; s)$ and $\zeta_{\eta}^*(\widehat{\phi};s)$ 
satisfy the following functional equation:
\begin{multline}
\label{form:FEofSiegelZeta}
\begin{pmatrix}
\zeta_{+}\left(\phi; \frac{m}{2}-s\right) \\[4pt]
\zeta_{-}\left(\phi; \frac{m}{2}-s\right)
\end{pmatrix}=
\Gamma\left(s+1-\frac{m}{2}\right)\Gamma(s) |D|^{\frac{1}{2}}
\cdot 2^{-2s+\frac{m}{2}}\cdot \pi^{-2s+\frac{m}{2}-1} \\
\times
\begin{pmatrix}
\sin \pi\left(\frac{p}{2}-s\right) & 
\sin \frac{\pi (m-p)}{2} \\[4pt]
\sin \frac{\pi p}{2} & 
\sin \pi\left(\frac{m-p}{2}-s\right)
\end{pmatrix}
\begin{pmatrix}
\zeta_{+}^*(\widehat{\phi};s) \\[4pt]
\zeta_{-}^*(\widehat{\phi};s)
\end{pmatrix}.
\end{multline}
\item The residues of 
$\zeta_{\epsilon}(\phi; s), \zeta_{\eta}^*(\widehat{\phi};s)$
at $s=1$ and $s=\frac{m}{2}$ are given by
\begin{align}
\label{form:ResidueZetaS=m/2}
\Res_{s=\frac{m}{2}}\zeta_{\epsilon}(\phi;s) &=
 \widehat{\phi}(0) 
  \int_{SO(Y)_{\R}/SO(Y)_{\Z}} d^{1}g, \\ 
\label{form:ResidueStarZetaS=m/2}
\Res_{s=\frac{m}{2}}\zeta_{\eta}^*(\widehat{\phi};s) 
 &=\phi(0) 
  \int_{SO(Y)_{\R}/SO(Y)_{\Z}} d^{1}g, \\[10pt]
 \label{form:ResidueZetaS=1}
\Res_{s=1}\zeta_{\epsilon}(\phi;s) &=
\Gamma\left(\frac{m}{2}-1\right) |D|^{\frac{1}{2}}
\cdot 2^{2-\frac{m}{2}}\cdot \pi^{1-\frac{m}{2}}
  \sum_{v^*\in SO(Y)_{\Z}\backslash S_{1, \Q}^{*}}
  \widehat{\phi}(v^*) \sigma^*(v^*) \\ \nonumber
  &\qquad \quad
 \times
 \left\{
 \begin{array}{ll}
\sin \dfrac{\pi}{2}(m-p) & (\epsilon= +) \\[10pt]
\sin \dfrac{\pi p}{2} & (\epsilon= -)
 \end{array}
 \right.,
   \\[10pt]
    \label{form:ResidueZetaStarS=1}
\Res_{s=1}\zeta_{\eta}^*(\widehat{\phi};s) &=
\Gamma\left(\frac{m}{2}-1\right) |D|^{-\frac{1}{2}}
\cdot 2^{2-\frac{m}{2}}\cdot \pi^{1-\frac{m}{2}}
  \sum_{v\in SO(Y)_{\Z}\backslash S_{1, \Q}}
  {\phi}(v) \sigma(v) \\ \nonumber
   &\qquad \quad
 \times
 \left\{
 \begin{array}{ll}
\sin \dfrac{\pi}{2}(m-p) & (\eta= +) \\[10pt]
\sin \dfrac{\pi p}{2} & (\eta= -)
 \end{array}
 \right..
\end{align}
\item The following relations hold:
\begin{align}
\label{form:SpecialVallueZetaAtm/2-1}
\zeta_{+}\left(\phi;\frac{m}{2}-1\right)+
\zeta_{-}\left(\phi;\frac{m}{2}-1\right) &= 
- \sum_{v\in SO(Y)_{\Z}\backslash S_{1, \Q}}
{\phi}(v) \sigma(v), 
\\[5pt]
\label{form:SpecialVallueZetaStarAtm/2-1}
\zeta_{+}^*\left(\widehat{\phi}\, ;\frac{m}{2}-1\right)+
\zeta_{-}^*\left(\widehat{\phi}\, ;\frac{m}{2}-1\right) &= 
-\sum_{v^*\in SO(Y)_{\Z}\backslash S_{1, \Q}^{*}}
  \widehat{\phi}(v^*) \sigma^*(v^*). 
\end{align}
\end{enumerate}
\end{theorem}

\begin{proof}
Let $f\in C_{0}^{\infty}(V_{\R})$ in 
Lemma~\ref{lemma:ACandFEofZetaIntegral} (1).
Then we see that 
\begin{align*}
 Z(f, \phi; s) &= Z_{+}(f, \phi; s)+
 Z_{+}^{*}\left(\widehat{f}, \widehat{\phi}; \frac{m}{2}-s\right)\\
 &\qquad +
   \frac{1}{s-1}\int_{S_{1,\R}^{*}} \widehat{f}(z^*) |\omega^*(z^*)|_{\infty}
  \sum_{v^*\in SO(Y)_{\Z}\backslash S_{1, \Q}^{*}}
  \widehat{\phi}(v^*) \sigma^*(v^*) \\
  &\qquad +\frac{\widehat{\phi}(0) \widehat{f}(0)}{s-\frac{m}{2}}
  \int_{SO(Y)_{\R}/SO(Y)_{\Z}} d^{1}g,
\end{align*}
and thus the integral $Z(f, \phi;s)$ can be continued to a meromorphic function 
on the whole $\C$, and ($s-1)(s-\frac{m}{2})Z(f, \phi;s)$ is an entire function of $s$.
Further, for any $s\in \C$, we take $f_{\epsilon}\in C_{0}^{\infty}(V_{\epsilon})$
such that $\Phi_{\epsilon}(f;s) \neq 0$. Then
Lemma~\ref{lemma:IntegralRepnOfZeta} implies that 
\[
 \zeta_{\epsilon}(\phi; s)= \frac{Z(\phi, f_{\epsilon}; s)}{\Phi_{\epsilon}(f;s)},
\]
and hence $\zeta_{\epsilon}(\phi; s)$ also can be continued to a meromorphic function 
on the whole $\C$, and ($s-1)(s-\frac{m}{2})\zeta_{\epsilon}(\phi;s)$ is an entire function of $s$.
The analytic continuation of $\zeta_{\eta}^*(\widehat{\phi};s)$ can be proved in a similar fashion. 
Further, one can prove the boundedness of $\zeta_{\epsilon}(\phi; s)$ and $\zeta_{\eta}^*(\widehat{\phi};s)$ 
in the same method as in Ueno~\cite[\S\ 4]{Ueno}.
By Lemma~\ref{lemma:IntegralRepnOfZeta} and 
Lemma~\ref{lemma:ACandFEofZetaIntegral}~(3), we have 
\[
\left(\Phi_{+}^*(\widehat{f} ;s) \; 
\Phi_{-}^*(\widehat{f} ;s)
\right)
\begin{pmatrix}
\zeta_{+}^*(\widehat{\phi} ; s) \\[4pt]
\zeta_{-}^*(\widehat{\phi}; s)
\end{pmatrix} 
= 
\left(\Phi_{+}\left(f; \tfrac{m}{2}-s\right) \;  
\Phi_{-}\left(f; \tfrac{m}{2}-s\right) 
\right)
\begin{pmatrix}
\zeta_{+}\left(\phi; \frac{m}{2}-s\right) \\[4pt]
\zeta_{-}\left(\phi; \frac{m}{2}-s\right) 
\end{pmatrix},
\]
and by Lemma~\ref{lemma:LFEfromGS}, we have
\[
\left(\Phi_{+}^*(\widehat{f} ;s) \; 
\Phi_{-}^*(\widehat{f} ;s)
\right) = 
\left(\Phi_{+}\left(f; \tfrac{m}{2}-s\right) \;  
\Phi_{-}\left(f; \tfrac{m}{2}-s\right) 
\right)\cdot {}^{t}A(s),
\]
where $A(s)$ is given by
\[
A(s)=
\Gamma\left(s+1-\frac{m}{2}\right)\Gamma(s) |D|^{\frac{1}{2}}
\cdot 2^{-2s+\frac{m}{2}}\cdot \pi^{-2s+\frac{m}{2}-1} 
\begin{pmatrix}
\sin \pi\left(\frac{p}{2}-s\right) & 
\sin \frac{\pi p}{2} \\[4pt]
\sin \frac{\pi(m-p)}{2} & 
\sin \pi\left(\frac{m-p}{2}-s\right)
\end{pmatrix}.
\]
This implies that the vector 
\begin{equation}
 \label{form:ZetaFunctionVector}
\begin{pmatrix}
\zeta_{+}\left(\phi; \frac{m}{2}-s\right) \\[4pt]
\zeta_{-}\left(\phi; \frac{m}{2}-s\right)
\end{pmatrix}- {}^{t}A(s)
\begin{pmatrix}
\zeta_{+}^*(\widehat{\phi};s) \\[4pt]
\zeta_{-}^*(\widehat{\phi};s)
\end{pmatrix}
\end{equation}
is orthogonal to the vector 
\[
\left(\Phi_{+}\left(f; \tfrac{m}{2}-s\right) \;  
\Phi_{-}\left(f; \tfrac{m}{2}-s\right) 
\right)
\]
for arbitrary $f\in \mathcal{S}(V_{\R})$.
For any $s\in \C$, there exists an 
$f_{\epsilon}\in C_{0}^{\infty}(V_{\epsilon})$ such that 
$\Phi_{\epsilon}(f; \frac{m}{2}-s)\neq 0$, and hence
\eqref{form:ZetaFunctionVector} is the zero vector. 
This proves the functional equation~\eqref{form:FEofSiegelZeta}.
Next we calculate the residues.
For the simple  pole at $s=\frac{m}{2}$, we have 
\[
 \Res_{s=\frac{m}{2}}Z(\phi, f;s) = 
  \widehat{\phi}(0) \widehat{f}(0) \cdot 
  \int_{SO(Y)_{\R}/SO(Y)_{\Z}} d^{1}g 
\] 
by Lemma~\ref{lemma:ACandFEofZetaIntegral}~(1).
For $f\in C_0^{\infty}(V_{\epsilon})$, 
Lemma~\ref{lemma:IntegralRepnOfZeta} implies
$Z(\phi, f;s)=\zeta_{\epsilon}(\phi;s) \cdot 
 \Phi_{\epsilon}\left(f; s\right)$, and  
 $\Phi_{\epsilon}\left(f; \tfrac{m}{2}\right)$ is meaningful:
\[
\Phi_{\epsilon}\left(f; \tfrac{m}{2}\right) 
=\lim_{s\rightarrow \frac{m}{2}}
\int_{V_{\epsilon}}
 f(x) |P(x)|^{s-\frac{m}{2}} dx = \int_{V_{\R}} f(x) dx= \widehat{f}(0).
\]
Hence we have
\[
\Res_{s=\frac{m}{2}}\zeta_{\epsilon}(\phi;s) =
 \widehat{\phi}(0) 
  \int_{SO(Y)_{\R}/SO(Y)_{\Z}} d^{1}g, 
\]
and similarly
\[
\Res_{s=\frac{m}{2}}\zeta_{\eta}^*(\widehat{\phi};s) =\phi(0) 
  \int_{SO(Y)_{\R}/SO(Y)_{\Z}} d^{1}g.
\]
By Lemma~\ref{lemma:ACandFEofZetaIntegral} (1),
it is easy to pick up
the residue of 
$Z(f, \phi; s)$ at the simple pole $s=1$, and
together with
Lemma~\ref{lemma:IntegralRepnOfZeta}, it implies that for 
$f\in C_0^{\infty}(V_{\epsilon})$,
\[
\Res_{s=1}\zeta_{\epsilon}(\phi;s) \cdot 
 \Phi_{\epsilon}\left(f; 1\right) =
\int_{S_{1,\R}^{*}} \widehat{f}(z^*) |\omega^*(z^*)|_{\infty}
  \sum_{v^*\in SO(Y)_{\Z}\backslash S_{1, \Q}^{*}}
  \widehat{\phi}(v^*) \sigma^*(v^*).
\]
Here the value
$\Phi_{\epsilon}\left(f; 1\right)$ is meaningful, and
\[
\Phi_{\epsilon}\left(f; 1\right) 
=\lim_{s\rightarrow 1}
\int_{V_{\epsilon}}
 f(x) |P(x)|^{s-\frac{m}{2}} dx = 
\int_{V_{\epsilon}}
 f(x) |P(x)|^{1-\frac{m}{2}} dx.  
\]
Furthermore, by Lemma~\ref{lemma:SingularInvariantDistributions}~(1), 
we have
\begin{multline*}
\int_{S_{1,\R}^*}\widehat{f}(v^*) |\omega^*(v^*)|_{\infty} 
=
\Gamma\left(\frac{m}{2}-1\right) |D|^{\frac{1}{2}}
\cdot 2^{2-\frac{m}{2}}\cdot \pi^{1-\frac{m}{2}} \\
\times
\left\{
\begin{array}{ll}
\displaystyle
\sin \dfrac{\pi}{2}(m-p)
\int_{V_{+}} f(v) |P(v)|^{1-\frac{m}{2}} dv 
&\qquad (f\in C_{0}^{\infty}(V_{+})
\\[15pt]
\displaystyle
\sin \dfrac{\pi p}{2} 
\int_{V_{-}} f(v) |P(v)|^{1-\frac{m}{2}} dv
&\qquad (f\in C_{0}^{\infty}(V_{-})
\end{array},
\right.
\end{multline*}
and hence we obtain the residue formula~\eqref{form:ResidueZetaS=1}.
%
Similarly, the residue formula~\eqref{form:ResidueZetaStarS=1}
can be proved
with Lemma~\ref{lemma:SingularInvariantDistributions}~(2);
the detail is omitted. 
To prove the relation~\eqref{form:SpecialVallueZetaAtm/2-1}, 
we let $s=1$ in the functiona equation~\eqref{form:FEofSiegelZeta}:
\begin{multline*}
\Gamma\left(s+1-\frac{m}{2}\right)^{-1}
\left(\zeta_{+}\left(\phi; \frac{m}{2}-s\right) 
+\zeta_{-}\left(\phi; \frac{m}{2}-s\right)\right) 
=
\Gamma(s) |D|^{\frac{1}{2}}
\cdot 2^{-2s+\frac{m}{2}}\cdot \pi^{-2s+\frac{m}{2}-1} \\[5pt]
\times
\begin{pmatrix}
\sin \pi\left(\frac{p}{2}-s\right) + 
\sin \frac{\pi p}{2} \quad &
\sin \frac{\pi (m-p)}{2} +
\sin \pi\left(\frac{m-p}{2}-s\right)
\end{pmatrix}
\begin{pmatrix}
\zeta_{+}^*(\widehat{\phi};s) \\[4pt]
\zeta_{-}^*(\widehat{\phi};s)
\end{pmatrix}.
\end{multline*}
Since
\[
\left(
\sin \pi\left(\frac{p}{2}-s\right) + 
\sin \frac{\pi p}{2}\right)\bigg|_{s=1} = 0, \quad 
\left(\sin \frac{\pi (m-p)}{2} +
\sin \pi\left(\frac{m-p}{2}-s\right)\right)
\bigg|_{s=1} = 0,
\]
we have
\begin{multline*}
\lim_{s\rightarrow1}
\Gamma\left(s+1-\frac{m}{2}\right)^{-1}
\left(\zeta_{+}\left(\phi; \frac{m}{2}-s\right) 
+\zeta_{-}\left(\phi; \frac{m}{2}-s\right) \right) \\[5pt]
=
|D|^{\frac{1}{2}}
\cdot 2^{-2+\frac{m}{2}}\cdot \pi^{-3+\frac{m}{2}} \\[5pt]
\times
\begin{pmatrix}
\frac{d}{ds} \sin\pi \left(\frac{p}{2}-s\right)
\bigg|_{s=1} \quad &
\frac{d}{ds} 
\sin \pi\left(\frac{m-p}{2}-s\right)
\bigg|_{s=1} 
\end{pmatrix}
\begin{pmatrix}
\Res_{s=1} \zeta_{+}^*(\widehat{\phi};s) \\[4pt]
\Res_{s=1} \zeta_{-}^*(\widehat{\phi};s)
\end{pmatrix}.
\end{multline*}
By using \eqref{form:ResidueZetaStarS=1}
and
\begin{align*}
\frac{d}{ds} \sin\pi \left(\frac{p}{2}-s\right)
\bigg|_{s=1} &= -\pi \cos\pi \left(\frac{p}{2}-1\right) = 
\pi \cos \frac{\pi p}{2}, \\
\frac{d}{ds} 
\sin \pi\left(\frac{m-p}{2}-s\right)
\bigg|_{s=1} &=-\pi \cos
\pi\left(\frac{m-p}{2}-1\right) =
\pi \cos \frac{\pi (m-p)}{2},
\end{align*}
we see that
\begin{align*}
&\quad 
\lim_{s\rightarrow 1}
\Gamma\left(s+1-\frac{m}{2}\right)^{-1}
\left(\zeta_{+}\left(\phi; \frac{m}{2}-s\right) 
+\zeta_{-}\left(\phi; \frac{m}{2}-s\right) \right)\\[5pt]
&=
\Gamma\left(\frac{m}{2}-1\right)
\cdot \pi^{-2}   \sum_{v\in SO(Y)_{\Z}\backslash S_{1, \Q}}
{\phi}(v) \sigma(v)
\cdot \begin{pmatrix}
\pi \cos \frac{\pi p}{2} \quad & 
\pi \cos \frac{\pi (m-p)}{2}
\end{pmatrix}
\begin{pmatrix}
\sin \frac{\pi (m-p)}{2}\\[5pt]
\sin \frac{\pi p}{2}
\end{pmatrix} \\[5pt]
&= -\Gamma\left(\frac{m}{2}-1\right)\sin\pi
\left(\frac{m}{2}-1\right) \cdot \pi^{-1} \cdot 
\sum_{v\in SO(Y)_{\Z}\backslash S_{1, \Q}}
{\phi}(v) \sigma(v).
\end{align*}
Since
\[
\lim_{s\rightarrow 1}
\Gamma\left(s+1-\frac{m}{2}\right)^{-1}
=\Gamma\left(\frac{m}{2}-1\right)\sin \pi\left(\frac{m}{2}-1\right)
\cdot \pi^{-1},
\]
we obtain the desired relation
\[
\zeta_{+}\left(\phi;\frac{m}{2}-1\right)+
\zeta_{-}\left(\phi;\frac{m}{2}-1\right) = 
- \sum_{v\in SO(Y)_{\Z}\backslash S_{1, \Q}}
{\phi}(v) \sigma(v). 
\]
Finally, let $s=\frac{m}{2}-1$ 
in the functional equation~\eqref{form:FEofSiegelZeta}.
We have
\begin{multline*}
-\Res_{s=1}
\begin{pmatrix}
\zeta_{+}\left(\phi; s\right) \\[4pt]
\zeta_{-}\left(\phi; s\right)
\end{pmatrix}=
\Res_{s=0}
\Gamma(s)\cdot \Gamma\left(\frac{m}{2}-1\right)
 |D|^{\frac{1}{2}}
\cdot 2^{-\frac{m}{2}+2}\cdot \pi^{-\frac{m}{2}+1} \\
\times
\begin{pmatrix}
\sin \pi\left(\frac{p}{2}-\frac{m}{2}+1\right) & 
\sin \frac{\pi (m-p)}{2} \\[4pt]
\sin \frac{\pi p}{2} & 
\sin \pi\left(-\frac{p}{2}+1\right)
\end{pmatrix}
\begin{pmatrix}
\zeta_{+}^*\left(\widehat{\phi}; \frac{m}{2}-1\right) \\[4pt]
\zeta_{-}^*\left(\widehat{\phi};\frac{m}{2}-1\right)
\end{pmatrix},
\end{multline*}
and by using \eqref{form:ResidueZetaS=1}, we obtain
\[
\zeta_{+}^*\left(\widehat{\phi}\, ;\frac{m}{2}-1\right)+
\zeta_{-}^*\left(\widehat{\phi}\, ;\frac{m}{2}-1\right) =
-\sum_{v^*\in SO(Y)_{\Z}\backslash S_{1, \Q}^{*}}
  \widehat{\phi}(v^*) \sigma^*(v^*). 
\]
\end{proof}

Let $N$ be the level of $2Y$. By definition, $N$
is the smallest positive integer such that $N(2Y)^{-1}$ is
an even matrix (a matrix whose entries are integers and even along
the diagonal).
We normalize the zeta functions
$\zeta_{\epsilon}(\phi; s), \zeta_{\eta}^*(\widehat{\phi};s)$
as follows:
\begin{align}
\label{form:ModifiedZeta}
\widetilde{\zeta}_{\epsilon}(\phi;s) &=
|D|^{-\frac{1}{2}}\cdot e^{\frac{\pi i}{4}(2p-m)} \cdot
\zeta_{\epsilon}\left(\phi; s+\frac{m}{2}-1\right), \\
\label{form:ModifiedZetaStar}
\widetilde{\zeta}_{\eta}^*(\widehat{\phi};s) &=N^{-s}\cdot 
\zeta_{\eta}^*\left(\widehat{\phi}; s+\frac{m}{2}-1\right).
\end{align}

\begin{lemma}
\label{lemma:ModifiedZeta}
The normalized zeta functions 
$\widetilde{\zeta}_{\epsilon}(\phi;s)$,  
$\widetilde{\zeta}_{\eta}^*(\widehat{\phi};s)$ satisfy the following
functional equation:
\begin{multline}
\label{form:ModifiedFE}
(2\pi)^{-s}\Gamma(s)
\gamma(s)
\begin{pmatrix}
\widetilde{\zeta}_{+}(\phi;s) \\[5pt]
\widetilde{\zeta}_{-}(\phi;s) 
\end{pmatrix} \\
=
N^{2-\frac{m}{2}-s}\cdot 
(2\pi)^{-(2-\frac{m}{2}-s)}\Gamma\left(2-\frac{m}{2}-s\right)\\
\times
\Sigma(2p-m)
\gamma\left(2-\frac{m}{2}-s\right)
\begin{pmatrix}
\widetilde{\zeta}_{+}^*\left(\widehat{\phi};2-\tfrac{m}{2}-s\right) \\[5pt]
\widetilde{\zeta}_{-}^*\left(\widehat{\phi};2-\tfrac{m}{2}-s\right) 
\end{pmatrix},
\end{multline}
where $\gamma(s)$ and $\Sigma(\ell)$ are matrices defined by
\eqref{form:DefOfGammaAndSigma}.
\end{lemma}

\begin{proof}
Let $s\mapsto 1-s$ in the functional equation~\eqref{form:FEofSiegelZeta}:
\begin{multline}
\label{form:ShiftedFE}
\begin{pmatrix}
\zeta_{+}\left(\phi; s+\frac{m}{2}-1\right) \\[4pt]
\zeta_{-}\left(\phi; s+\frac{m}{2}-1\right)
\end{pmatrix}=
\Gamma\left(2-\frac{m}{2}-s\right)\Gamma(1-s) |D|^{\frac{1}{2}}
\cdot 2^{2s+\frac{m}{2}-2}\cdot \pi^{2s+\frac{m}{2}-3} \\
\times
\begin{pmatrix}
\sin \pi\left(s+\frac{p}{2}-1\right) & 
\sin \frac{\pi (m-p)}{2} \\[4pt]
\sin \frac{\pi p}{2} & 
\sin \pi\left(s+\frac{m-p}{2}-1\right)
\end{pmatrix}
\begin{pmatrix}
\zeta_{+}^*(\widehat{\phi};1-s) \\[4pt]
\zeta_{-}^*(\widehat{\phi};1-s)
\end{pmatrix}.
\end{multline}
By \eqref{form:ModifiedZeta}, \eqref{form:ModifiedZetaStar}
and $\displaystyle \Gamma(1-s)  =\frac{\pi}{\Gamma(s)\sin \pi s}$, 
we see that this relation can be written as 
\begin{multline*}
(2\pi)^{-s}\Gamma(s)
\begin{pmatrix}
\widetilde{\zeta}_{+}(\phi;s) \\[5pt]
\widetilde{\zeta}_{-}(\phi;s) 
\end{pmatrix} 
= N^{2-\frac{m}{2}-s}\cdot 
(2\pi)^{-(2-\frac{m}{2}-s)} \Gamma\left(2-\frac{m}{2}-s\right)\\
\times
\frac{e^{\frac{\pi i}{4}(2p-m)}}{\sin \pi s}
\begin{pmatrix}
\sin \pi\left(s+\frac{p}{2}-1\right) & 
\sin \frac{\pi (m-p)}{2} \\[4pt]
\sin \frac{\pi p}{2} & 
\sin \pi\left(s+\frac{m-p}{2}-1\right)
\end{pmatrix}
\\
\times
\begin{pmatrix}
\widetilde{\zeta}_{+}^*\left(\widehat{\phi};2-\tfrac{m}{2}-s\right) \\[5pt]
\widetilde{\zeta}_{-}^*\left(\widehat{\phi};2-\tfrac{m}{2}-s\right) 
\end{pmatrix}.
\end{multline*}
An elementary calculation with
$\det\gamma(s) =2i \sin \pi s$ shows that 
\[
\frac{e^{\frac{\pi i}{4}(2p-m)}}{\sin \pi s}
\begin{pmatrix}
\sin \pi\left(s+\frac{p}{2}-1\right) & 
\sin \frac{\pi (m-p)}{2} \\[4pt]
\sin \frac{\pi p}{2} & 
\sin \pi\left(s+\frac{m-p}{2}-1\right)
\end{pmatrix}=
\gamma(s)^{-1} \cdot \Sigma(2p-m)
\gamma\left(2-\frac{m}{2}-s\right),
\]
which completes the proof of the lemma. 
\end{proof}

The functional equation~\eqref{form:ModifiedFE} 
is quite the same as the functional equation of
the condition~{\bf [A3]} in
\S~\ref{section:Preliminaries}
with $\frac{m}{2}=2\lambda$, $\ell\equiv 2p-m \pmod{4}$.
Hence it is reasonable to expect that 
our converse theorem (Lemma~\ref{corollary:Maassforms})
 can apply to the normalized zeta functions 
$\widetilde{\zeta}_{\epsilon}(\phi;s)$,  
$\widetilde{\zeta}_{\eta}^*(\widehat{\phi};s)$ 
to obtain Maass forms. 
The following lemma is indispensable for 
the application. 

\begin{lemma}
\label{corollary:ZeroCondition}
\begin{enumerate}
\item If $m$ is odd, then we have
\[
 \widetilde{\zeta}_{+}(\phi; -k) + (-1)^{k} \cdot \widetilde{\zeta}_{-}(\phi; -k)=0
\]
for $k=1,2,3, \dots$.
\item Assume that $m$ is even and $p$ is odd.
Let $q=\frac{m}{2}$. Then we have
\[
 \widetilde{\zeta}_{+}(\phi; -k) + (-1)^{k} \cdot \widetilde{\zeta}_{-}(\phi; -k)=0
\]
for $k=1, 2, \dots, q-2$.
\end{enumerate}
\end{lemma}

\begin{proof}
By a little calculation, we obtain
\begin{multline}
\begin{pmatrix}
\zeta_{+}^*\left(\widehat{\phi}\, ; 1-s\right) \\[4pt]
\zeta_{-}^*\left(\widehat{\phi}\, ; 1-s\right)
\end{pmatrix}=
\Gamma(s)
\Gamma\left(s+\frac{m}{2}-1\right)|D|^{-\frac{1}{2}}
\cdot 2^{-2s-\frac{m}{2}+2}\cdot \pi^{-2s-\frac{m}{2}+1} \\
\times
\begin{pmatrix}
\sin \pi\left(s+\frac{m-p}{2}\right) & 
\sin \frac{\pi (m-p)}{2} \\[4pt]
\sin \frac{\pi p}{2} & 
\sin \pi\left(s+\frac{p}{2}\right)
\end{pmatrix}
\begin{pmatrix}
\zeta_{+}\left(\phi; s+\frac{m}{2}-1\right) \\[4pt]
\zeta_{-}\left(\phi; s+\frac{m}{2}-1\right)
\end{pmatrix}.
\end{multline}
Let us consider the values of both sides at 
$s=-k \; (k\in \Z_{> 0})$.
On the left hand side, 
$\zeta_{\eta}^*\left(\widehat{\phi}\, ; 1-s\right)$
is holomorphic at $s=-k$ except when
$m$ is even and $k=\frac{m}{2}-1=q-1$. 
On the right hand side, if $m$ is odd, then
$\Gamma(s)\Gamma\left(s+\frac{m}{2}-1\right)$ has a simple
pole at $s=-k \; (k\in \Z_{>0})$, and 
if $m$ is even, then 
$\Gamma(s)\Gamma\left(s+\frac{m}{2}-1\right)
=\Gamma(s)\Gamma(s+q-1)$ has a simple pole at
$s=-k$ ($1\leq k\leq q-2$). 
We assume that $1\leq k\leq q-2$ in the case of even $m$. 
Then, since $\Gamma(s)\Gamma\left(s+\frac{m}{2}-1\right)$
has a simple pole at $s=-k$, we see that 
\begin{align*}
& \quad \begin{pmatrix}
\sin \pi\left(s+\frac{m-p}{2}\right) \cdot
\zeta_{+}\left(\phi; s+\frac{m}{2}-1\right) +
\sin \frac{\pi (m-p)}{2} \cdot
\zeta_{-}\left(\phi; s+\frac{m}{2}-1\right) 
\\[4pt]
\sin \frac{\pi p}{2} \cdot
\zeta_{+}\left(\phi; s+\frac{m}{2}-1\right) + 
\sin \pi\left(s+\frac{p}{2}\right)\cdot 
\zeta_{-}\left(\phi; s+\frac{m}{2}-1\right) 
\end{pmatrix} \\[8pt]
&=
\begin{pmatrix}
\sin \pi\left(s+\frac{m-p}{2}\right) 
\cdot  \widetilde{\zeta}_{+}(\phi; s) 
+\sin \frac{\pi (m-p)}{2} 
\cdot  \widetilde{\zeta}_{-}(\phi; s) 
\\[4pt]
\sin \frac{\pi p}{2} 
\cdot  \widetilde{\zeta}_{+}(\phi; s) 
+\sin \pi\left(s+\frac{p}{2}\right)
\cdot  \widetilde{\zeta}_{-}(\phi; s) 
\end{pmatrix} 
\end{align*}
becomes the zero vector at $s=-k$. 
Since $\sin \pi (-k)=0$, $\cos \pi (-k)=(-1)^{k}$, we have
\begin{align*}
(-1)^k \sin  \frac{\pi (m-p)}{2}
\cdot  \widetilde{\zeta}_{+}(\phi; -k) 
+\sin \frac{\pi (m-p)}{2} 
\cdot  \widetilde{\zeta}_{-}(\phi; -k) &=0,
\\[4pt]
\sin \frac{\pi p}{2} 
\cdot  \widetilde{\zeta}_{+}(\phi; -k) 
+(-1)^{k}\sin  \frac{\pi p}{2}
\cdot  \widetilde{\zeta}_{-}(\phi; -k) &=0.
\end{align*}
If $m$ is odd, then either $p$ or $m-p$ is odd, and 
thus we have
\[
 \widetilde{\zeta}_{+}(\phi; -k) + (-1)^{k} \cdot \widetilde{\zeta}_{-}(\phi; -k)=0.
\]
In the case of even $m$, if $p$ is odd, then the relation above should hold.
In the case that both of $p$ and $m-p$ are even, this argument can not apply
since $\sin\frac{\pi p}{2}= \sin \frac{\pi (m-p)}{2}=0$.
\end{proof}

The following lemma follows immediately from 
the relations 
\eqref{form:SpecialVallueZetaAtm/2-1} and
\eqref{form:SpecialVallueZetaStarAtm/2-1}.
\begin{lemma}
We have the following relations:
\begin{align}
\label{form:ZetaValueAt0}
-\left( \widetilde{\zeta}_{+}(\phi; 0)+
\widetilde{\zeta}_{-}(\phi; 0)\right) &=
|D|^{-\frac{1}{2}}\cdot e^{\frac{\pi i}{4}(2p-m)} \cdot 
\sum_{v\in SO(Y)_{\Z}\backslash S_{1, \Q}}
  {\phi}(v) \sigma(v), \\
\label{form:ZetaStarValueAt0}
  -\left( \widetilde{\zeta}_{+}^*(\widehat{\phi}\; ; 0)+
\widetilde{\zeta}_{-}^*(\widehat{\phi}\, ; 0)\right) &=
 \sum_{v^*\in SO(Y)_{\Z}\backslash S_{1, \Q}^{*}}
  \widehat{\phi}(v^*) \sigma^*(v^*).
\end{align}
\end{lemma}
In the rest of this section, we discuss the invariance of volumes
with respect to scalar multiplications. 
\begin{lemma}
\label{lemma:ScalarMultInv}
\begin{enumerate}
\item 
For 
$v\in V_{\Q}-S_{\Q}, v^*\in V_{\Q}-S_{\Q}^*$, 
we define the volumes
$\mu(v)$ and $\mu^*(v^*)$ by
\eqref{form:DefOfMuV} and
\eqref{form:DefOfMustarVstar}, respectively. 
For $r>0$, we have
\[
  \mu(r v) = \mu(v), \qquad \mu^*(r v^*) = \mu^*(v^*).
\]
\item For $v\in S_{1, \Q}, v^*\in S_{1, \Q}^*$, 
we define the volumes 
$\sigma(v)$ and $\sigma^*(v^*)$
by
\eqref{form:DefOfNV} and
\eqref{form:DefOfNVStar}, respectively. 
For $r>0$, we have
\[
 \sigma(r v)= r^{2-m} \cdot \sigma(v), \qquad 
 \sigma^*(r v^*) = r^{2-m} \cdot \sigma^*(v^*).
\]
\end{enumerate}
\end{lemma}

\begin{proof}
(1) We prove the second formula
$\mu^*(r v^*) = \mu^*(v^*)$, which will be used later. 
Let $F\in C_{0}^{\infty}(V_{\eta}^*)$. Then, by
\eqref{form:NormalizationGvStar} and 
\eqref{form:DefOfMustarVstar}, we have
\begin{align*}
&\quad \int_0^{\infty} d^{\times}t
\int_{SO(Y)_{\R}/SO(Y)_{v^*, \Z}}
F(\rho^*(t, g)v^*) d^{1} g \\
&=
\int_{V_{\eta}^{*}} F(x^*) |P^*(x^*)|^{-\frac{m}{2}} dx^*
\int_{SO(Y)_{v^*, \R}/SO(Y)_{v^*, \Z}}d \mu_{v^*}^{*}(h)\\
&= \mu^*(v^*)\cdot 
\int_{V_{\eta}^{*}} F(x^*) |P^*(x^*)|^{-\frac{m}{2}} dx^*. 
\end{align*}
By the substitution $v^*\mapsto r v^*$, we have
\[
\int_0^{\infty} d^{\times}t
\int_{SO(Y)_{\R}/SO(Y)_{rv^*, \Z}}
F(\rho^*(t, g)\cdot rv^*) d^{1} g \\
= \mu^*(rv^*)\cdot 
\int_{V_{\eta}^{*}} F(x^*) |P^*(x^*)|^{-\frac{m}{2}} dx^*. 
\]
Put $F_{r}(v^*):=F(r v^*)$. Since $SO(Y)_{r v^*}= SO(Y)_{v^*}$, we have
\begin{align*}
\int_0^{\infty} d^{\times}t
\int_{SO(Y)_{\R}/SO(Y)_{rv^*, \Z}}
F(\rho^*(t, g) \cdot r v^*) d^{1} g 
&=
\int_0^{\infty} d^{\times}t
\int_{SO(Y)_{\R}/SO(Y)_{v^*, \Z}}
F_{r}(\rho^*(t, g) v^*) d^{1} g \\
&= \mu^*(v^*)\cdot 
\int_{V_{\eta}^{*}} F_r(x^*) |P^*(x^*)|^{-\frac{m}{2}} dx^* \\
&=\mu^*(v^*)\cdot 
\int_{V_{\eta}^{*}} F(r x^*) |P^*(x^*)|^{-\frac{m}{2}} dx^* \\
&=\mu^*(v^*)\cdot 
\int_{V_{\eta}^{*}} F(x^*) 
\underbrace{|P^*(r^{-1} x^*)|^{-\frac{m}{2}} d(r^{-1} x^* )}_{=
|P^*(x^*)|^{-\frac{m}{2}} dx^*}\\
&=\mu^*(v^*)\cdot 
\int_{V_{\eta}^{*}} F(x^*) |P^*(x^*)|^{-\frac{m}{2}} dx^*.
\end{align*}
This proves $\mu^{*}(r v^*)= \mu^*(v^*)$.
The first formula can be proved similarly.

\bigskip

\noindent
(2) Let us show that
 $\sigma(r v)= r^{2-m} \cdot \sigma(v)$.
Take an $f\in \Sch{V_{\R}}$ and put
$\psi(g)= f(gv)$. By using \eqref{form:NormalizationDnu}, 
we have
\begin{align*}
\int_{SO(Y)_{\R}/SO(Y)_{v,\Z}} f(gv) d^{1} g &=
\int_{SO(Y)_{\R}/SO(Y)_{v, \R}} f(gv) |\omega(\dot{g}v)|_{\infty}
\int_{SO(Y)_{v, \R}/SO(Y)_{v, \Z}}  d\sigma_{v}(h) \\
&= \sigma(v) \cdot 
\int_{S_{1,\R}} f(z) |\omega(z)|_{\infty}.
\end{align*}
By the substitution $v\mapsto r v$, we have
\[
\int_{SO(Y)_{\R}/SO(Y)_{rv,\Z}} f(grv) d^{1} g =
 \sigma(r v) \cdot 
\int_{S_{1,\R}} f(z) |\omega(z)|_{\infty}.
\]
Put $f_{r}(v):= f(rv)$. Since $SO(Y)_{rv}=SO(Y)_{v}$, we have
\begin{align*}
 \int_{SO(Y)_{\R}/SO(Y)_{rv,\Z}} f(grv) d^{1} g&= 
\int_{SO(Y)_{\R}/SO(Y)_{v,\Z}} f_r(gv) d^{1} g \\
&=
 \sigma(v) \cdot 
\int_{S_{1,\R}} f_r(z) |\omega(z)|_{\infty} \\
&=
 \sigma(v) \cdot 
\int_{S_{1,\R}} f(rz) |\omega(z)|_{\infty} \\
&=
 \sigma(v) \cdot 
\int_{S_{1,\R}} f(z) |\omega(r^{-1} z)|_{\infty} \\
&=
 \sigma(v) \cdot 
\int_{S_{1,\R}} f(z) r^{-(m-2)}
|\omega(z)|_{\infty} \\
&=  r^{2-m}\cdot \sigma(v) \cdot 
\int_{S_{1,\R}} f(z) |\omega(z)|_{\infty},
\end{align*}
where we have used \eqref{form:InvarianceOmega}
on the fifth equality.
This proves 
$\sigma(rv) = r^{2-m}\cdot \sigma(v)$.
The second formula can be proved in a similar fashion.  
\end{proof}

\section{The main theorem}
\label{section:MainTh}

To prove the functional equation of twisted zeta functions, we quote
a result of Stark~\cite{Stark}.
Let $Y$ be a non-degenerate half-integral symmetric matrix of degree $m$.
Let $D=\det(2Y)$ and $N$ be the level of $2Y$. 
We define a half-integral symmetric matrix $\widehat{Y}$ by 
\[
\widehat{Y}=\frac{1}{4}N Y^{-1}.
\] 
We define the quadratic form $P(v)$ on $V$ by $P(v)=Y[v]=
{}^{t} v Yv$, and the quadratic form $\widehat{P}(v^*)$
on $V^*$ by
\begin{equation}
\label{form:DefOfHatP}
\widehat{P}(v^*) = \widehat{Y}[v^*] =NP^*(v^*),
\end{equation}
where $P^*$ is defined by \eqref{form:DefOfPstar}.
For this $\widehat{P}$, we define the measure 
$M^*(\widehat{P};n)$ of representation by 
\begin{equation}
 \label{form:MhatPn}
M^*(\widehat{P};\pm n) =
\sum
\begin{Sb}
v^*\in SO(Y)_{\Z}\backslash V_{\pm}^*\cap V_{\Z} \\
\widehat{P}(v^*)=\pm n
\end{Sb} \mu^*(v^*).
\end{equation}
For an odd prime $r$ with $(r, N)=1$ and a Dirichlet character $\psi$ of modulus $r$, we 
define the function $\phi_{\psi,P}(v)$ on $V_{\Q}$ by
\[
 \phi_{\psi, P}(v) = \tau_{\psi}(P(v)) \cdot \phi_0(v), 
\]
where $\tau_{\psi}(P(v))$  is the Gauss sum defined by 
\eqref{eqn:def of Gauss sum},
 and $\phi_0(v)$ is the characteristic function of 
$\Z^m$. It is easy to see that 
$\phi_{\psi, P}(v)$ is a Schwartz-Bruhat function on $V_{\Q}$.
We define a field $K$ by
\[
 K=
 \begin{cases}
 \Q(\sqrt{(-1)^{m/2} D})    & (m\equiv 0 \pmod{2}) \\
 \Q(\sqrt{2|D|})    & (m\equiv 1 \pmod{2}) 
 \end{cases},
\]
and $\chi_{K}$ be the Kronecker symbol associated to $K$.
(If $K=\Q$, we regard $\chi_K$ as the principal character.)
Furthermore, we define a Dirichlet character $\psi^*\bmod{r}$ by
\[
 \psi^*(k)= \overline{\psi(k)}\left(\frac{k}{r}\right)^{m}, 
\]
and  put
\[
 C_{2p-m,r} =
  \begin{cases}
  1 &  (m\equiv 0 \pmod{2})\\
  \varepsilon_{r}^{2p-m}  &   (m\equiv 1 \pmod{2})
 \end{cases}
\]
as \eqref{eqn:def of clr}.
Then the following lemma follows from 
Stark~\cite[Lemmas~5 and~6]{Stark}.
\begin{lemma}
\label{lemma:StarkModified}
Let $\widehat{\phi_{\psi, P}}(v^*)$
be the Fourier transform 
of $\phi_{\psi, P}$ 
defined by \eqref{form:FourierTransformOverQ}.
Then the support of $\widehat{\phi_{\psi, P}}(v^*)$ is contained in
$r^{-1}\Z^{m}$, and for $v^{*}\in \Z^m$, we have
\[
\widehat{\phi_{\psi, P}}(r^{-1} v^*) =
r^{-m/2} \chi_{K}(r)\cdot C_{2p-m, r}\cdot
\psi^*(-N)\cdot \tau_{\psi^*}(\widehat{P}(v^*)).
\]
\end{lemma}

Let $\phi=\phi_0$ in the normalized zeta function
$\widetilde{\zeta}_{\pm}(\phi;s)$ of
\eqref{form:ModifiedZeta}.
For $v\in V_{\epsilon}\cap V_{\Z}$, we have
$P(v)=\epsilon n$ for some $n=1,2,3,\dots$, and hence
$\widetilde{\zeta}_{\pm}(\phi_0;s)$ can be transformed as
\begin{align*}
\widetilde{\zeta}_{\pm}(\phi_0;s) &=
|D|^{-\frac{1}{2}}\cdot e^{\frac{\pi i}{4}(2p-m)} \cdot
\zeta_{\pm}\left(\phi_0; s+\frac{m}{2}-1\right)\\
 &=|D|^{-\frac{1}{2}}\cdot e^{\frac{\pi i}{4}(2p-m)} \cdot
\sum_{v\in SO(Y)_{\Z}\backslash V_{\pm}\cap V_{\Z}}
\frac{\mu(v)}{|P(v)|^{s+\frac{m}{2}-1}} \\
 &=|D|^{-\frac{1}{2}}\cdot e^{\frac{\pi i}{4}(2p-m)} \cdot
\sum_{n=1}^{\infty}
\left\{
\sum
\begin{Sb}
v\in SO(Y)_{\Z}\backslash V_{\pm}\cap V_{\Z} \\
P(v)=\pm n
\end{Sb} \mu(v) \right\} 
n^{-s-\frac{m}{2}-1}\\
&=\sum_{n=1}^{\infty} \frac{\mathbf{a}(\pm n)}{n^s},
\end{align*}
where $\mathbf{a}(\pm n) \; (n=1,2,3,\dots)$ is defined by
\begin{align}
\label{form:DefOfFrakA}
\mathbf{a}(\pm n) &=
|D|^{-\frac{1}{2}}\cdot e^{\frac{\pi i}{4}(2p-m)} \cdot
n^{1-\frac{m}{2}}
\sum
\begin{Sb}
v\in SO(Y)_{\Z}\backslash V_{\pm}\cap V_{\Z} \\
P(v)=\pm n
\end{Sb} \mu(v) \\ \nonumber
&=|D|^{-\frac{1}{2}}\cdot e^{\frac{\pi i}{4}(2p-m)} \cdot
n^{1-\frac{m}{2}} \cdot M(P; \pm n),
\end{align}
where $M(P; n)$ is the measure of representation defined as  
\eqref{form:DefOfMPn}.
Further, by plugging $\phi_{\psi, P}(v) = \tau_{\psi}(P(v)) \cdot \phi_0(v)$
in \eqref{form:ModifiedZeta}, we have
\begin{align*}
\widetilde{\zeta}_{\pm}(\phi_{\psi, P};s) 
 &=|D|^{-\frac{1}{2}}\cdot e^{\frac{\pi i}{4}(2p-m)} \cdot
\sum_{v\in SO(Y)_{\Z}\backslash V_{\pm}\cap V_{\Z}}
\frac{\tau_{\psi}(P(v)) \mu(v)}{|P(v)|^{s+\frac{m}{2}-1}} \\
&=\sum_{n=1}^{\infty} \frac{\tau_{\psi}(\pm n)\mathbf{a}(\pm n)}{n^s}.
\end{align*}
On the other hand, let $\phi=\phi_0$ in the normalized 
zeta function $\widetilde{\zeta}_{\eta}^*(\widehat{\phi};s)$
of \eqref{form:ModifiedZetaStar}.
Since $\widehat{\phi_0}= \phi_0$, we have
\begin{align*}
\widetilde{\zeta}_{\pm}^*(\widehat{\phi_0};s) &=N^{-s}\cdot 
\zeta_{\pm}^*\left(\phi_0; s+\frac{m}{2}-1\right) \\
&= N^{-s}\cdot 
\sum_{v^*\in SO(Y)_{\Z}\backslash V_{\pm}^*\cap V_{\Z}}
\frac{\mu^*(v^*)}{|P^*(v^*)|^{s+\frac{m}{2}-1}} \\
&= N^{\frac{m}{2}-1}
\cdot 
\sum_{v^*\in SO(Y)_{\Z}\backslash V_{\pm}^*\cap V_{\Z}}
\frac{\mu^*(v^*)}{|NP^*(v^*)|^{s+\frac{m}{2}-1}}.
\end{align*}
By the definition~\eqref{form:DefOfHatP}, 
we have $\widehat{P}(v^*)= NP^*(v^*) \in \Z\setminus \{0\}$
for $v^*\in V_{\pm}^{*}\cap V_{\Z}$ and 
hence
\begin{align*}
\widetilde{\zeta}_{\pm}^*(\widehat{\phi_0};s) &=
N^{\frac{m}{2}-1}\cdot 
\sum_{n=1}^{\infty} \left\{\sum
\begin{Sb}
v^*\in SO(Y)_{\Z}\backslash V_{\pm}^*\cap V_{\Z} \\
\widehat{P}(v^*)=\pm n
\end{Sb} \mu^*(v^*)\right\} n^{-s-\frac{m}{2}-1} \\
&= \sum_{n=1}^{\infty} \frac{\mathbf{b}(\pm n)}{n^s},
\end{align*}
where $\mathbf{b}(\pm n) \; (n=1,2,3,\dots)$ is defined by
\begin{equation}
\label{form:DefOfFrakB}
\mathbf{b}(\pm n) = 
\left(\frac{n}{N}\right)^{1-\frac{m}{2}}
\sum
\begin{Sb}
v^*\in SO(Y)_{\Z}\backslash V_{\pm}^*\cap V_{\Z} \\
\widehat{P}(v^*)=\pm n
\end{Sb} \mu^*(v^*) = 
\left(\frac{n}{N}\right)^{1-\frac{m}{2}}\cdot 
M^*(\widehat{P}; \pm n), 
\end{equation}
where $M^*(\widehat{P}; n)$ is defined as \eqref{form:MhatPn}.
Finally, let $\phi=\phi_{\psi, P}(v) = \tau_{\psi}(P(v)) \cdot \phi_0(v)$
in \eqref{form:ModifiedZetaStar}.
It then follows from Lemmas~\ref{lemma:ScalarMultInv}~(1) 
and~\ref{lemma:StarkModified}, 
and also 
$\widehat{P}(r^{-1} v^*)= r^{-2} \cdot 
\widehat{P}(v^*)$ that 
\begin{align*}
\widetilde{\zeta}_{\eta}^*(\widehat{\phi_{\psi, P}};s) 
&= N^{\frac{m}{2}-1}
\cdot 
\sum_{v^*\in SO(Y)_{\Z}\backslash V_{\pm}^*\cap V_{\Z}}
\frac{\widehat{\phi_{\psi, P}}(r^{-1} v^*)
\mu^*(r^{-1} v^*)}{|NP^*(r^{-1} v^*)|^{s+\frac{m}{2}-1}} \\
&=N^{\frac{m}{2}-1}
\cdot
r^{-m/2} \chi_{K}(r)\cdot C_{2p-m, r}\cdot
\psi^*(-N)\\
&\qquad \quad \cdot r^{2(s+\frac{m}{2}-1)}
\cdot 
\sum_{v^*\in SO(Y)_{\Z}\backslash V_{\pm}^*\cap V_{\Z}}
\frac{\tau_{\psi^*}(\widehat{P}(v^*))
\mu^*(v^*)}{|\widehat{P}(v^*)|^{s+\frac{m}{2}-1}} \\
&=r^{2s+\frac{m}{2}-2}\chi_{K}(r)\cdot C_{2p-m, r}\cdot
\psi^*(-N) \cdot
\sum_{n=1}^{\infty} \frac{\tau_{\psi^*}(\pm n) \mathbf{b}(\pm n)}{n^s}.
\end{align*}
We thus obtain the first assertion of the following 
\begin{lemma}
\label{lemma:VerifyTheAssumption}
For $n=1,2,3,\dots$, we define 
$\mathbf{a}(\pm n)$ and $\mathbf{b}(\pm n)$
by \eqref{form:DefOfFrakA} and  \eqref{form:DefOfFrakB} 
respectively, and let
\begin{align*}
\zeta_{\pm}(\mathbf{a}; s) &=
\sum_{n=1}^{\infty} \frac{\mathbf{a}(\pm n)}{n^s},  & 
\zeta_{\pm}(\mathbf{a}, \psi; s) &=
\sum_{n=1}^{\infty} \frac{\tau_{\psi}(\pm n)\mathbf{a}(\pm n)}{n^s},  \\
\zeta_{\pm}(\mathbf{b}; s) &=
\sum_{n=1}^{\infty} \frac{\mathbf{b}(\pm n)}{n^s},  & 
\zeta_{\pm}(\mathbf{b}, \psi^*; s) &=
\sum_{n=1}^{\infty} \frac{\tau_{\psi^*}(\pm n)\mathbf{b}(\pm n)}{n^s}.
\end{align*}
\begin{enumerate}
\item We have
\begin{align*}
\widetilde{\zeta}_{\pm}(\phi_0;s) &= \zeta_{\pm}(\mathbf{a};s),  \\
\widetilde{\zeta}_{\pm}(\phi_{\psi, P};s) 
&=\zeta_{\pm}(\mathbf{a}, \psi; s), \\
\widetilde{\zeta}_{\eta}^*(\widehat{\phi_0};s) &=
\zeta_{\pm}(\mathbf{b}; s),  \\
\widetilde{\zeta}_{\eta}^*(\widehat{\phi_{\psi, P}};s) &=
r^{2s+\frac{m}{2}-2}\chi_{K}(r)\cdot C_{2p-m, r}\cdot
\psi^*(-N) \cdot
\zeta_{\pm}(\mathbf{b}, \psi^*; s). 
\end{align*}
\item On residues and special values of zeta functions, 
the following four relations hold:
\begin{align}
\label{form:FirstRelation}
&
{\zeta}_{+}(\mathbf{a}, \psi; 0)+
{\zeta}_{-}(\mathbf{a}, \psi; 0) =
 \tau_{\psi}(0)\cdot \left({\zeta}_{+}(\mathbf{a}; 0)+
{\zeta}_{-}(\mathbf{a}; 0)\right), \\[3pt]
\label{form:SecondRelation}
&r^{\frac{m}{2}} \cdot \chi_{K}(r)\cdot C_{2p-m, r}\cdot
\psi^*(-N) \cdot \Res_{s=1}
\zeta_{\pm}(\mathbf{b}, \psi^*; s)=
\tau_{\psi}(0) \Res_{s=1}\zeta_{\pm}(\mathbf{b}; s), \\[3pt]
\label{form:ThirdRelation}
&{\zeta}_{+}(\mathbf{b}, \psi^*; 0)+
{\zeta}_{-}(\mathbf{b}, \psi^*; 0)=
\tau_{\psi^*}(0) \cdot 
\left({\zeta}_{+}(\mathbf{b} ; 0)+
{\zeta}_{-}(\mathbf{b}\, ; 0)\right), \\[3pt]
\label{form:FourthRelation}
&\Res_{s=1}
\zeta_{\pm}(\mathbf{a}, \psi; s)=
r^{-\frac{m}{2}}\cdot \chi_{K}(r)\cdot C_{2p-m, r}\cdot
\psi^*(-N)\cdot 
\Res_{s=1}\zeta_{\pm}(\mathbf{a}; s).
\end{align}
\item
Assume that at least one of $m$ or $p$ is an odd integer.
Let $\lambda=\frac{m}{4}$ and take an integer $\ell$
with $\ell\equiv 2p-m \pmod{4}$.
Then $\zeta_{\pm}(\mathbf{a}; s)$ and $\zeta_{\pm}(\mathbf{b};s)$
satisfy the assumptions {\bf [A1]}--{\bf [A4]} of 
\S \ref{section:Preliminaries}, and further, 
$\zeta_{\pm}(\mathbf{a}, \psi; s)$ and $\zeta_{\pm}(\mathbf{b}, \psi^*;s)$
satisfy the assumptions {\bf [A1]}${}_{r,\psi}$--{\bf [A5]}${}_{r, \psi}$
of \S \ref{section:Preliminaries}.

\end{enumerate}
\end{lemma}

\begin{proof}
(2) By letting $\phi=\phi_0$ in \eqref{form:ZetaValueAt0}, 
we have
\begin{equation}
 \label{form:ZetaFrakA0}
-\left({\zeta}_{+}(\mathbf{a}; 0)+
{\zeta}_{-}(\mathbf{a}; 0)\right) =
|D|^{-\frac{1}{2}}\cdot e^{\frac{\pi i}{4}(2p-m)} \cdot 
\sum_{v\in SO(Y)_{\Z}\backslash S_{1, \Z}}
   \sigma(v),
\end{equation}
and by letting $\phi=\phi_{\psi, P}$
in \eqref{form:ZetaValueAt0}, we have
\begin{align*}   
-\left({\zeta}_{+}(\mathbf{a}, \psi; 0)+
{\zeta}_{-}(\mathbf{a}, \psi; 0)\right) &=
|D|^{-\frac{1}{2}}\cdot e^{\frac{\pi i}{4}(2p-m)} \cdot 
\tau_{\psi}(0)
\sum_{v\in SO(Y)_{\Z}\backslash S_{1, \Z}}
  \sigma(v) \\
  &=
  -\tau_{\psi}(0)\cdot \left({\zeta}_{+}(\mathbf{a}; 0)+
{\zeta}_{-}(\mathbf{a}; 0)\right),
\end{align*}
which proves \eqref{form:FirstRelation}.
By~\eqref{form:ModifiedZetaStar} and
Theorem~\ref{theorem:SiegelZetaProperties}~(3), we have
\begin{align*}
\Res_{s=1}\zeta_{\pm}(\mathbf{b}; s) &= 
\Res_{s=1}
\widetilde{\zeta}_{\pm}^*(\widehat{\phi_0};s) \\
&= \Res_{s=1}\left(
N^{-s}\cdot 
\zeta_{\pm}^*\left(\widehat{\phi_0}; s+\frac{m}{2}-1\right)\right)
\\
&= N^{-1}\cdot \Res_{s=\frac{m}{2}} \zeta_{\pm}^{*}
(\widehat{\phi_0}; s) \\
&= N^{-1}
  \int_{SO(Y)_{\R}/SO(Y)_{\Z}} d^{1}g, 
\end{align*}
and we thus obtain
\begin{equation}
\label{form:ResZetaFrakBat1}
\Res_{s=1}\zeta_{\pm}(\mathbf{b}; s) = 
N^{-1}
  \int_{SO(Y)_{\R}/SO(Y)_{\Z}} d^{1}g.
\end{equation}
Let us consider the residues at $s=1$ of the both sides of
\[
\widetilde{\zeta}_{\pm}^*(\widehat{\phi_{\psi, P}};s) =
r^{2s+\frac{m}{2}-2}\chi_{K}(r)\cdot C_{2p-m, r}\cdot
\psi^*(-N) \cdot
\zeta_{\pm}(\mathbf{b}, \psi^*; s).
\]
The residue at $s=1$ of the left hand side is
\[
\Res_{s=1}
\widetilde{\zeta}_{\pm}^*(\widehat{\phi_{\psi, P}};s) 
=N^{-1} \cdot \phi_{\psi, P}(0)  \int_{SO(Y)_{\R}/SO(Y)_{\Z}} d^{1}g 
= \tau_{\psi}(0) \Res_{s=1}\zeta_{\pm}(\mathbf{b}; s),
\]
and that of the right hand side is
\begin{align*}
&\qquad \Res_{s=1}\left\{r^{2s+\frac{m}{2}-2}\chi_{K}(r)\cdot C_{2p-m, r}\cdot
\psi^*(-N) \cdot
\zeta_{\pm}(\mathbf{b}, \psi^*; s)\right\} \\
&= r^{\frac{m}{2}} \cdot \chi_{K}(r)\cdot C_{2p-m, r}\cdot
\psi^*(-N) \cdot \Res_{s=1}
\zeta_{\pm}(\mathbf{b}, \psi^*; s),
\end{align*}
by which we obtain~\eqref{form:SecondRelation}.
Next let $\phi=\phi_0$ in  the relation~\eqref{form:ZetaStarValueAt0}.
Then we have
\begin{equation}
 \label{form:ZetaFrakB0}
-\left({\zeta}_{+}(\mathbf{b} ; 0)+
{\zeta}_{-}(\mathbf{b}\, ; 0)\right) =
 \sum_{v^*\in SO(Y)_{\Z}\backslash S_{1, \Z}^{*}}
  \sigma^*(v^*).
\end{equation}
By letting $\phi=\phi_{\psi, P}$ in~\eqref{form:ZetaStarValueAt0}
and using Lemmas~\ref{lemma:StarkModified}
and~\ref{lemma:ScalarMultInv}~(2), we have
\begin{align*}
&\qquad
-r^{\frac{m}{2}-2}\chi_{K}(r)\cdot C_{2p-m, r}\cdot
\psi^*(-N) \cdot
\left({\zeta}_{+}(\mathbf{b}, \psi^*; 0)+
{\zeta}_{-}(\mathbf{b}, \psi^*; 0)\right) \\
&=  \sum_{v^*\in SO(Y)_{\Z}\backslash S_{1, \Z}^{*}}
  \widehat{\phi_{\psi, P}}(r^{-1} v^*)
  \underbrace{\sigma^*(r^{-1} v^*)}_{=r^{m-2} \cdot \sigma^*(v^*)} \\
&=r^{-\frac{m}{2}} \chi_{K}(r)\cdot C_{2p-m, r}\cdot
\psi^*(-N)\cdot \tau_{\psi^*}(0) \cdot r^{m-2}
\sum_{v^*\in SO(Y)_{\Z}\backslash S_{1, \Z}^{*}}
\sigma^*(v^*) \\
&=- r^{\frac{m}{2}-2} \chi_{K}(r)\cdot C_{2p-m, r}\cdot
\psi^*(-N)\cdot \tau_{\psi^*}(0) 
\left({\zeta}_{+}(\mathbf{b} ; 0)+
{\zeta}_{-}(\mathbf{b}\, ; 0)\right),
\end{align*}
and this proves 
\[
{\zeta}_{+}(\mathbf{b}, \psi^*; 0)+
{\zeta}_{-}(\mathbf{b}, \psi^*; 0)=
\tau_{\psi^*}(0) \cdot 
\left({\zeta}_{+}(\mathbf{b} ; 0)+
{\zeta}_{-}(\mathbf{b}\, ; 0)\right),
\]
which is the relation~\eqref{form:ThirdRelation}.
By \eqref{form:ModifiedZeta} and 
Theorem~\ref{theorem:SiegelZetaProperties}~(3), we have
\begin{align*}
\Res_{s=1}\zeta_{\pm}(\mathbf{a}; s) &= 
\Res_{s=1}
\widetilde{\zeta}_{\pm}({\phi_0};s) \\
&= \Res_{s=1}\left(
|D|^{-\frac{1}{2}}\cdot e^{\frac{\pi i}{4}(2p-m)} \cdot
\zeta_{\pm}\left({\phi_0}; s+\frac{m}{2}-1\right)\right)
\\
&= 
|D|^{-\frac{1}{2}}\cdot e^{\frac{\pi i}{4}(2p-m)} 
\cdot \Res_{s=\frac{m}{2}} \zeta_{\pm}
\left({\phi_0}; s\right) \\
&= |D|^{-\frac{1}{2}}\cdot e^{\frac{\pi i}{4}(2p-m)}
  \int_{SO(Y)_{\R}/SO(Y)_{\Z}} d^{1}g,
\end{align*}
and we thus obtain
\begin{equation}
\label{form:ResZetaFrakAat1}
\Res_{s=1}\zeta_{\pm}(\mathbf{a}; s) = 
|D|^{-\frac{1}{2}}\cdot e^{\frac{\pi i}{4}(2p-m)}
  \int_{SO(Y)_{\R}/SO(Y)_{\Z}} d^{1}g.
\end{equation}
Furthermore, it follows from  Lemma~\ref{lemma:StarkModified} 
that the residue of
$\zeta_{\pm}(\mathbf{a}, \psi; s)=
\widetilde{\zeta}_{\pm}(\phi_{\psi, P};s)$
at $s=1$ is given by 
\begin{align*}
\Res_{s=1}\zeta_{\pm}(\mathbf{a}, \psi; s)
&=|D|^{-\frac{1}{2}}\cdot e^{\frac{\pi i}{4}(2p-m)} \cdot
\widehat{\phi_{\psi, P}}(0)
  \int_{SO(Y)_{\R}/SO(Y)_{\Z}} d^{1}g \\
  &=
  |D|^{-\frac{1}{2}}\cdot e^{\frac{\pi i}{4}(2p-m)} \cdot
  r^{-\frac{m}{2}} \chi_{K}(r)\cdot C_{2p-m, r}\cdot
\psi^*(-N)
  \int_{SO(Y)_{\R}/SO(Y)_{\Z}} d^{1}g \\
  &= r^{-\frac{m}{2}} \chi_{K}(r)\cdot C_{2p-m, r}\cdot
\psi^*(-N)\cdot 
\Res_{s=1}\zeta_{\pm}(\mathbf{a}; s), 
\end{align*}
by which we obtain the relation~\eqref{form:FourthRelation}.

\medskip

\noindent
(3) By Theorem~\ref{theorem:SiegelZetaProperties}~(1), (3), 
we see that our zeta functions satisfy the assumptions 
{\bf [A1]}, {\bf [A1]}${}_{r, \psi}$, 
{\bf [A2]}, and {\bf [A2]}${}_{r, \psi}$.
The functional equation of {\bf [A3]} is nothing but 
the equation~\eqref{form:ModifiedFE} with $\phi=\phi_0$.
Let $\phi=\phi_{\psi, P}$ in~\eqref{form:ModifiedFE};
then the first assertion of the lemma implies that  
\begin{multline*}
(2\pi)^{-s}\Gamma(s)
\gamma(s)
\begin{pmatrix}
{\zeta}_{+}(\mathbf{a}, \psi ;s) \\[5pt]
{\zeta}_{-}(\mathbf{a}, \psi;s) 
\end{pmatrix} \\
=
\chi_{K}(r)\cdot C_{2p-m, r}\cdot
\psi^*(-N) \cdot r^{\frac{m}{2}-2} 
\cdot (Nr^2)^{2-\frac{m}{2}-s} \\
\cdot 
(2\pi)^{-(2-\frac{m}{2}-s)}\Gamma\left(2-\frac{m}{2}-s\right) \\
\cdot
\Sigma(2p-m)
\gamma\left(2-\frac{m}{2}-s\right)
\begin{pmatrix}
{\zeta}_{+}\left(\mathbf{b}, \psi^*;2-\tfrac{m}{2}-s\right) \\[5pt]
{\zeta}_{-}\left(\mathbf{b}, \psi^*;2-\tfrac{m}{2}-s\right) 
\end{pmatrix},
\end{multline*}
which shows that  the functional equation of {\bf [A3]}${}_{r, \psi}$ holds.
Lemma~\ref{corollary:ZeroCondition} implies that our zeta functions
satisfy the assumptions~{\bf [A4]} and
{\bf [A4]}${}_{r, \psi}$.
Finally, the compatibility condition~{\bf [A5]}${}_{r, \psi}$ on residues 
and special values follows from \eqref{form:FirstRelation}, 
\eqref{form:SecondRelation}, \eqref{form:ThirdRelation} and 
\eqref{form:FourthRelation}.
\end{proof}
In general,  $SO(Y)_{\Z}\backslash S_{1, \Z}$ 
is always an infinite set, 
since for $v\in S_{1, \Z}$, any two of $v, 2v, 3v,\dots$ can not lie in the same $SO(Y)_{\Z}$-orbit. However, as is seen in
the following lemma that is taken from~\cite[pp.188--189]{PVBook},
the number of $SO(Y)_{\Z}$-orbits in {\it primitive} vectors in 
$S_{1, \Z}$ and $S_{1, \Z}^*$ 
is finite.
\begin{lemma}
\label{lemma:FromPVBook}
\begin{enumerate}
\item We call a vector $v=(v_1, \dots, v_{m})\in V_{\Z}$ primitive if the greatest common
divisor of $v_1,\dots, v_m$ is $1$. Then 
\[
\{v\in SO(Y)_{\Z}\backslash S_{1,\Z}\, ;\, v \, \textrm{is primitive}\}
\]
is a finite set. Let $a_1, \dots, a_h$ be a complete system of 
representatives of this set. Then we have 
\[
\sum_{v\in SO(Y)_{\Z}\backslash S_{1, \Z}}
   \sigma(v) = \zeta(m-2) \sum_{i=1}^{h} \sigma(a_i).
\]
\item Let $b_1, \dots, b_k$ be a complete system of
the finite set 
\[
\{v^*\in SO(Y)_{\Z}\backslash S_{1,\Z}^*\, ;\, v^* \, \textrm{is primitive}\}.
\]
Then we have
\[
\sum_{v\in SO(Y)_{\Z}\backslash S_{1, \Z}^*}
   \sigma^*(v^*) = \zeta(m-2) \sum_{i=1}^{k} \sigma^*(b_i).
\]
\end{enumerate}
\end{lemma}
Now we are in a position to state
\begin{theorem}
\label{theorem:SiegelMaass}
Assume that at least one of $m$ or $p$ is an odd integer.
Take an integer $\ell$
with $\ell\equiv 2p-m \pmod{4}$.
Define $C^{\infty}$-functions 
$F(z)$ and $G(z)$ on $\mathcal{H}$ by
\begin{align*}
F(z) &=   y^{(m-\ell)/4}\cdot \int_{SO(Y)_{\R}/SO(Y)_{\Z}} d^{1}g  \\
&+
 (-1)^{(2p-m-\ell)/4}
\zeta(m-2) \sum_{i=1}^{h} \frac{\sigma(a_{i})}{|D|^{\frac{1}{2}}} 
\cdot
\frac{(2\pi) 2^{1-\frac{m}{2}} \Gamma(\frac{m}{2}-1)}
{\Gamma\left(\frac{m+\ell}{4}\right)
\Gamma\left(\frac{m-\ell}{4}\right)}\cdot y^{1-(m+\ell)/4}\\
&+\sum
\begin{Sb}
n=-\infty \\
n\neq 0
\end{Sb}^{\infty}
 (-1)^{(2p-m-\ell)/4}\cdot  \frac{M(P;n)}{|D|^{\frac{1}{2}}}
\frac{\pi^{\frac{m}{4}} \cdot  |n|^{-\frac{m}{4}}}
{\Gamma\left(\frac{m+\sgn(n)\ell}{4}\right)}\cdot
y^{-\frac{\ell}{4}}W_{\frac{\sgn(n)\ell}{4}, \frac{m}{4}-\frac{1}{2}}(4\pi |n|y) \bfe[nx], \\
G(z) 
 &=
 N^{\frac{m}{4}} \cdot |D|^{-\frac{1}{2}} e^{\frac{\pi i}{4}(2p-m)} 
\cdot y^{(m-\ell)/4} \cdot  \int_{SO(Y)_{\R}/SO(Y)_{\Z}} d^{1}g
\\
&+
i^{-\frac{\ell}{2}}
N^{1-\frac{m}{4}}\zeta(m-2) \sum_{i=1}^{k} \sigma^*(b_{i}) 
\frac{(2\pi) 2^{1-\frac{m}{2}} \Gamma(\frac{m}{2}-1)}
{\Gamma\left(\frac{m+\ell}{4}\right)
\Gamma\left(\frac{m-\ell}{4}\right)}\cdot y^{1-(m+\ell)/4}\\
&+i^{-\frac{\ell}{2}} \sum
\begin{Sb}
n=-\infty \\
n\neq 0
\end{Sb}^{\infty}
\left(\frac{|n|}{N}\right)^{-\frac{m}{4}}
M^*(\widehat{P};n)
\frac{\pi^{\frac{m}{4}}}
{\Gamma\left(\frac{m+\sgn(n)\ell}{4}\right)}\cdot
y^{-\frac{\ell}{4}}W_{\frac{\sgn(n)\ell}{4}, \frac{m}{4}-\frac{1}{2}}(4\pi |n|y) \bfe[nx].
\end{align*}
Then, $F(z)$ (resp. $G(z)$) is a Maass form for $\Gamma_0(N)$ of weight $\ell/2$
with eigenvalue $(m-\ell)(4-m-\ell)/16$ and
character $\chi_{K}$ (resp.\ $\chi_{K_N}$). Here we denote by $\chi_{K}$ and
$\chi_{K_{N}}$ the Kronecker characters associated to the fields 
\[
 K=
 \begin{cases}
 \Q(\sqrt{(-1)^{m/2} D})    & (m\equiv 0 \pmod{2}) \\
 \Q(\sqrt{2|D|})    & (m\equiv 1 \pmod{2}) 
 \end{cases},
\]
and
\[
 K_{N}=
 \begin{cases}
 \Q(\sqrt{(-1)^{m/2} D})    & (m\equiv 0 \pmod{2}) \\
 \Q(\sqrt{2|D|N})    & (m\equiv 1 \pmod{2}) 
 \end{cases},
\]
respectively. Further we have
\[
F\left(-\frac{1}{Nz}\right) 
(\sqrt{N} z)^{-\ell/2} = G(z).
\]
\end{theorem}

\begin{proof}
We apply the converse theorem
(Lemma~\ref{corollary:Maassforms}) to the normalized
zeta functions $\zeta_{\pm}(\mathbf{a};s)$
and $\zeta_{\pm}(\mathbf{b};s)$ of
Lemma~\ref{lemma:VerifyTheAssumption}. 
It remains to calculate the constant terms
$\mathbf{a}(0),\, \mathbf{a}(\infty), \,
\mathbf{b}(0), \, \mathbf{b}(\infty)$
along with the definitions 
\eqref{form:DefOfA0AndAInfty},
\eqref{form:DefOfA0AndAInfty2},
\eqref{form:DefOfB0AndBInfty},
\eqref{form:DefOfB0AndBInfty2}.
First, 
by \eqref{form:ZetaFrakA0} and Lemma~\ref{lemma:FromPVBook}~(1), we have
\begin{align*}
\mathbf{a}(0)&=
-\left({\zeta}_{+}(\mathbf{a}; 0)+
{\zeta}_{-}(\mathbf{a}; 0)\right) \\
&=
|D|^{-\frac{1}{2}}\cdot e^{\frac{\pi i}{4}(2p-m)} \cdot 
\zeta(m-2)
\sum_{i=1}^{h} \sigma(a_i).
\end{align*}
Second, by \eqref{form:ResZetaFrakBat1}, we have
\begin{align*}
 \mathbf{a}(\infty) &= \frac{N}{2} 
 \left(\Res_{s=1} \zeta_{+}(\mathbf{b};s)+
\Res_{s=1} \zeta_{-}(\mathbf{b};s)\right)\\
&=\int_{SO(Y)_{\R}/SO(Y)_{\Z}} d^{1}g.
\end{align*}
Third, 
by \eqref{form:ZetaFrakB0} and Lemma~\ref{lemma:FromPVBook}~(2), we have
\begin{align*}
\mathbf{b}(0)&=
-\left({\zeta}_{+}(\mathbf{b}; 0)+
{\zeta}_{-}(\mathbf{b}; 0)\right) \\
&= \zeta(m-2) \sum_{i=1}^{k} \sigma^*(b_i).
\end{align*}
Finally, by
\eqref{form:ResZetaFrakAat1}, we have
\begin{align*}
 \mathbf{b}(\infty) &= \frac{i^{-\ell}}{2} 
 \left(\Res_{s=1} \zeta_{+}(\mathbf{a};s)+
\Res_{s=1} \zeta_{-}(\mathbf{a};s)\right)\\
&=i^{-\ell}\cdot |D|^{-\frac{1}{2}} e^{\frac{\pi i}{4}(2p-m)}
\int_{SO(Y)_{\R}/SO(Y)_{\Z}} d^{1}g \\
&= |D|^{-\frac{1}{2}} e^{-\frac{\pi i}{4}(2p-m)}
\int_{SO(Y)_{\R}/SO(Y)_{\Z}} d^{1}g.
\end{align*}
\end{proof}

\begin{remark}
\label{remark:RelationToSiegel}
One can verify that our $M(P;n)$ is identical to $M(\mathfrak{S}, \mathfrak{a}, t)$ 
($\mathfrak{a}=0$), which is defined as the formula (14) of Siegel~\cite{Siegel}. 
Moreover, up to a power of $y$, our
$F(z)$ coincides
with the integral $\int_{F} f_{\mathfrak{a}}(z, \mathfrak{P})dv \; (\mathfrak{a}=0)$ 
of the indefinite theta series $f_{\mathfrak{a}}(z, \mathfrak{P})$ over some 
fundamental domain $F$. See Siegel~\cite[Hilfssatz 4]{Siegel} for the detail.
We also note that Funke~\cite{Funke} calculated the Mellin
transform of some indefinite theta series and obtained Siegel's zeta functions 
associated with ternary zero forms.
\end{remark}

\section{Holomorphic modular forms arising from Siegel's zeta functions}
\label{section:Holomorphic}

Under some conditions, the $\gamma$-matrix in Siegel's functional equation~\eqref{form:FEofSiegelZeta}
can be an upper or lower triangular matrix. 
In such a case, we obtain a single functional equation. More precisely, 

\medskip

\noindent
$\bullet$ Assume that the number of negative eigenvalues of $Y$ is even;
that is, $m-p$ is an even integer.
Then the first row of~\eqref{form:FEofSiegelZeta}
is of the following form:
\[
\zeta_{+}\left(\phi; \frac{m}{2}-s\right)=
\Gamma\left(s+1-\frac{m}{2}\right)\Gamma(s) |D|^{\frac{1}{2}}
\cdot 2^{-2s+\frac{m}{2}}\cdot \pi^{-2s+\frac{m}{2}-1} 
\sin \pi\left(\frac{p}{2}-s\right)\zeta_{+}^*(\widehat{\phi};s).  
\]
This suggests that $\zeta_{+}(\phi; s)$ and $\zeta_{+}^*(\phi;s)$ satisfy the functional equation of
Hecke type.

\bigskip

\noindent
$\bullet$ Assume that the number of positive eigenvalues of $Y$ is even;
that is, $p$ is an even integer.
Then the second row of~\eqref{form:FEofSiegelZeta}
is of the following form:
\[
\zeta_{-}\left(\phi; \frac{m}{2}-s\right)=
\Gamma\left(s+1-\frac{m}{2}\right)\Gamma(s) |D|^{\frac{1}{2}}
\cdot 2^{-2s+\frac{m}{2}}\cdot \pi^{-2s+\frac{m}{2}-1} 
\sin \pi\left(\frac{m-p}{2}-s\right)
\zeta_{-}^*(\widehat{\phi};s).
\]
This suggests that $\zeta_{-}(\phi; s)$ and $\zeta_{-}^*(\phi;s)$ satisfy the functional equation of
Hecke type.

\bigskip

In the following, we assume that $m-p$ is even; if $p$ is even, we replace
$P$ with $-P$. 
We introduce Dirichlet series $L(M;s)$ and $L(M^*; s)$ 
as follows: 
\begin{align}
\label{form:HolDSM}
L(M;s) &=\sum_{n=1}^{\infty} \frac{a(n)}{n^s} 
& \text{with}\quad a(n)&:=  |D|^{-1/2}\cdot M(P;n), \\
\label{form:HolDSMStar}
L(M^*;s) &=
\sum_{n=1}^{\infty} \frac{b(n)}{n^s}
& \text{with}\quad b(n) &:=(-1)^{\frac{m-2p}{4}}\cdot N^{\frac{m}{4}}\cdot M^*(\widehat{P};n).
\end{align}
Further, we put 
\begin{align*}
\Lambda_{N}(s; M) &=
\left(\frac{2\pi}{\sqrt{N}}\right)^{-s}\cdot \Gamma(s)\cdot L(M; s), \\
\Lambda_{N}(s; M^*) &=
\left(\frac{2\pi}{\sqrt{N}}\right)^{-s}\cdot \Gamma(s)\cdot L(M^*;s), 
\end{align*}
and
\begin{align}
a(0) &=(-1)^{\frac{m-p}{2}} (2\pi)^{-\frac{m}{2}} \cdot \Gamma\left(\frac{m}{2}\right) 
\int_{SO(Y)_{\R}/SO(Y)_{\Z}} d^{1}g, \\
b(0) &=i^{-\frac{m}{2}}\cdot (2\pi)^{-\frac{m}{2}} \cdot \Gamma\left(\frac{m}{2}\right) N^{\frac{m}{4}} |D|^{-1/2}
\int_{SO(Y)_{\R}/SO(Y)_{\Z}} d^{1}g.
\end{align}
Then Theorem~\ref{theorem:SiegelZetaProperties} implies that 
 the following lemma holds: 
%
\begin{lemma}
Assume that $m-p$ is even. 
Both $\Lambda_{N}(s; M)$ and $\Lambda_{N}(s; M^*)$ can be continued analytically to
the whole $s$-plane, satisfy the functional equation
\[
 \Lambda_{N}(s; M) = i^{\frac{m}{2}} \Lambda_{N}\left(\frac{m}{2}-s; M^*\right),
\]
and the function
\[
 \Lambda_{N}(s; M)+ \frac{a(0)}{s}+ \frac{i^{\frac{m}{2}}\cdot b(0)}{\frac{m}{2}-s}
\]
is holomorphic on the whole $s$-plane and bounded on any vertical strip.
\end{lemma}

Let $r$ be an odd prime with $(N, r)=1$.
We denote by $\varphi=\left(\frac{*}{r}\right)$ the Dirichlet character defined by
the quadratic residue symbol.
For a {\it primitive} Dirichlet character $\psi\bmod{r}$,  we define Dirichlet series
$L(M;s, \psi)$ and $L(M^*; s, \psi)$ by
\begin{align*}
L(M;s, \psi) &=\sum_{n=1}^{\infty} \frac{\psi(n) a(n)}{n^s}, \\
L(M^*;s,\psi) &=
\begin{cases}
\displaystyle
r\sum_{n=1}^{\infty} \frac{b(rn)}{(rn)^{s}} - \sum_{n=1}^{\infty} \frac{b(n)}{n^s}
& \text{if $m$ is odd and}\ \psi= \varphi=\left(\dfrac{\ast}{r}\right), \\[20pt]
\displaystyle
\sum_{n=1}^{\infty} \frac{\psi(n)b(n)}{n^s} & \text{otherwise},
\end{cases}
\end{align*}
where $a(n)$ and $b(n)$ are defined by \eqref{form:HolDSM} and 
\eqref{form:HolDSMStar}, respectively. 
Furthermore, we set
\begin{align*}
\Lambda_{N}(s; M, \psi) &=
\left(\frac{2\pi}{r\sqrt{N}}\right)^{-s}\cdot \Gamma(s)\cdot L(M; s, \psi), \\
\Lambda_{N}(s; M^*, \psi) &=
\left(\frac{2\pi}{r\sqrt{N}}\right)^{-s}\cdot \Gamma(s)\cdot L(M^*;s, \psi). 
\end{align*}
Then, by using Lemma~\ref{lemma:StarkModified}, 
the formulas~\eqref{eqn:Gauss sum non-trivial psi} 
and~\eqref{eqn:Gauss sum trivial psi}, we can prove the following
\begin{lemma}
Assume that $m-p$ is even. 
\begin{enumerate}
\item In the case of even $m$, for any primitive Dirichlet character $\psi\bmod{r}$, 
$\Lambda_{N}(s;M, \psi)$ can be holomorphically continued to the whole $s$-plane,
bounded on any vertical strip, and satisfies the following functional equation
\[
\Lambda_{N}(s; M, \psi)= i^{\frac{m}{2}} C_{\psi} \Lambda_{N}
\left(\frac{m}{2}-s; M^*, \overline{\psi}\right)
\]
with the constant
\[
 C_{\psi}= \chi_{K}(r) \psi(-N)\tau_{\psi}/\tau_{\overline{\psi}}.
\]
\item In the case of odd $m$, for any primitive Dirichlet character $\psi\mod{r}$
with $\psi\neq \varphi=\left(\frac{*}{r}\right)$, 
$\Lambda_{N}(s;M, \psi)$ can be holomorphically continued to the whole $s$-plane,
bounded on any vertical strip, and satisfies the following functional equation
\[
\Lambda_{N}(s; M, \psi)= i^{\frac{m}{2}} C_{\psi}^{(1)} \Lambda_{N}
\left(\frac{m}{2}-s; M^*, \overline{\psi}\varphi\right)
\]
with the constant
\[
 C_{\psi}^{(1)}= \left(\frac{-1}{m}\right)^{\frac{m-1}{2}} \cdot \chi_{K}(r)\left(\frac{N}{r}\right)
 \psi(-N)\epsilon_{r}^{-1} \tau_{\psi\varphi}/\tau_{\overline{\psi}}.
\]
\item In the case that $m$ is odd and $\psi= \varphi=\left(\frac{*}{r}\right)$,
\[
 \Lambda_{N}(s; M, \psi)+ C_{\psi}^{(2)}\frac{(r^{1/2}-r^{-1/2})b(0)}{\frac{m}{2}-s}
\]
can be holomorphically continued to the whole $s$-plane,
bounded on any vertical strip, and satisfies the following functional equation
\[
\Lambda_{N}(s; M, \varphi)= i^{\frac{m}{2}} C_{\psi}^{(2)} \Lambda_{N}
\left(\frac{m}{2}-s; M^*, \varphi\right)
\]
with the constant
\[
  C_{\psi}^{(2)}= \left(\frac{-1}{m}\right)^{\frac{m-1}{2}} \cdot \chi_{K}(r).
\]
\end{enumerate}
\end{lemma}

These lemmas show that  Weil's converse theorems for holomorphic modular forms can apply to 
$L(M;s)$ and $L(M^*;s)$. 
We refer to
Miyake~\cite[Theorem~4.3.15]{Miyake} 
for Weil's converse theorem for the case of integral weight. 
For the case of half-integral weight, 
Shimura~\cite{Shimura73} stated a similar converse theorem.
Although the details were not given in ~\cite{Shimura73}, the proof is roughly identical to 
the case of  integral weight, and can be found in 
Bruinier~\cite{Bruinier}.
We therefore obtain the following
\begin{theorem}
\label{theorem:HolomorphicForms}
Assume that $m-p$ is even. 
We define holomorphic functions $F(z)$ and $G(z)$ on $\mathcal{H}$ by
\begin{align*}
F(z) &= (-1)^{\frac{m-p}{2}} (2\pi)^{-\frac{m}{2}} \cdot \Gamma\left(\frac{m}{2}\right) 
\int_{SO(Y)_{\R}/SO(Y)_{\Z}} d^{1}g \\
&\qquad 
 +  |D|^{-1/2}\cdot\sum_{n=1}^{\infty}   M(P;n)\mathbf{e}[nz], \\
G(z) &= 
i^{-\frac{m}{2}}\cdot (2\pi)^{-\frac{m}{2}} \cdot \Gamma\left(\frac{m}{2}\right) N^{\frac{m}{4}} |D|^{-1/2}
\int_{SO(Y)_{\R}/SO(Y)_{\Z}} d^{1}g \\
&\qquad 
+ (-1)^{\frac{m-2p}{4}}\cdot N^{\frac{m}{4}}\cdot \sum_{n=1}^{\infty}
M^*(\widehat{P};n)\mathbf{e}[nz].
\end{align*}
Then, $F(z)$ (resp. $G(z)$) is a holomorphic modular  form for $\Gamma_0(N)$ of weight $m/2$
with character $\chi_{K}$ (resp.\ $\chi_{K_N}$).
Further we have
\[
F\left(-\frac{1}{Nz}\right) 
(\sqrt{N} z)^{-m/2} = G(z).
\]
\end{theorem}

\begin{remark}
If $p$ is even, we can prove the same assertion for $M(P;-n)$.
Theorem~\ref{theorem:SiegelMaass} excludes the case where both $m$ and $p$ are even, 
but Theorem~\ref{theorem:HolomorphicForms} shows that 
both $\zeta_+$ and $\zeta_-$ correspond to holomorphic modular forms
in this case. 
\end{remark}

%
%
%
%
%

\end{document}